\documentclass[12pt]{amsart}  

\usepackage[latin1]{inputenc}
\usepackage{amsmath} 
\usepackage{amsfonts}
\usepackage{amssymb}
\usepackage{stmaryrd}
\usepackage{latexsym} 
\usepackage{graphicx}
\usepackage{subfigure}
\usepackage{color}
\usepackage{hyperref}
\usepackage{verbatim}
\usepackage[all]{xy}
\usepackage{graphics}
\usepackage{pdfsync}
\usepackage{xcolor}
\usepackage{tikz}
\usepackage{bm}
\usetikzlibrary{arrows.meta}
\usetikzlibrary{shapes.geometric}

\theoremstyle{definition}
\newtheorem{theorem}{Theorem}[section]
\newtheorem{proposition}[theorem]{Proposition}

\newtheorem{lemma}[theorem]{Lemma}
\newtheorem{corollary}[theorem]{Corollary}
\newtheorem{conjecture}[theorem]{Conjecture}
\newtheorem{definition}[theorem]{Definition}
\newtheorem{example}[theorem]{Example}

\newtheorem{definition/theorem}[theorem]{Definition/Theorem}

\theoremstyle{remark}
\newtheorem{remark}[theorem]{Remark}

\numberwithin{equation}{section}
\newlength\cellsize 
\savebox2{%
\begin{picture}(15,15)
\put(0,0){\line(1,0){15}}
\put(0,0){\line(0,1){15}}
\put(15,0){\line(0,1){15}}
\put(0,15){\line(1,0){15}}
\end{picture}}
\newcommand\cellify[1]{\def\thearg{#1}\def\nothing{}%
\ifx\thearg\nothing
\vrule width0pt height\cellsize depth0pt\else
\hbox to 0pt{\usebox2\hss}\fi%
\vbox to 15\unitlength{
\vss
\hbox to 15\unitlength{\hss$#1$\hss}
\vss}}
\newcommand\tableau[1]{\vtop{\let\\=\cr
\setlength\baselineskip{-16000pt}
\setlength\lineskiplimit{16000pt}
\setlength\lineskip{0pt}
\halign{&\cellify{##}\cr#1\crcr}}}
\savebox3{%
\begin{picture}(15,15)
\put(0,0){\line(1,0){15}}
\put(0,0){\line(0,1){15}}
\put(15,0){\line(0,1){15}}
\put(0,15){\line(1,0){15}}
\end{picture}}
\newcommand\expath[1]{%
\hbox to 0pt{\usebox3\hss}%
\vbox to 15\unitlength{
\vss
\hbox to 15\unitlength{\hss$#1$\hss}
\vss}}
\newcommand\bas[1]{\omit \vbox to \cellsize{ \vss \hbox to \cellsize{\hss$#1$\hss} \vss}}

\begin{document}

\title[A combinatorial Schur expansion of triangle-free horizontal-strip LLT polynomials]{A combinatorial Schur expansion of triangle-free horizontal-strip LLT polynomials}

\author{Foster Tom}

\thanks{
The author was supported in part by the National Sciences and Engineering Research Council of Canada.}

\begin{abstract} In recent years, Alexandersson and others proved combinatorial formulas for the Schur function expansion of the horizontal-strip LLT polynomial $G_{\bm\lambda}(\bm x;q)$ in some special cases. We associate a weighted graph $\Pi$ to $\bm\lambda$ and we use it to express a linear relation among LLT polynomials. We apply this relation to prove an explicit combinatorial Schur-positive expansion of $G_{\bm\lambda}(\bm x;q)$ whenever $\Pi$ is triangle-free. We also prove that the largest power of $q$ in the LLT polynomial is the total edge weight of our graph.
\end{abstract}

\maketitle
\section{Introduction}\label{sec:intro}

LLT polynomials are remarkable symmetric functions with many connections in algebraic combinatorics. Lascoux, Leclerc, and Thibon \cite{lltoriginal} originally defined LLT polynomials in terms of ribbon tableaux in order to study Fock space representations of the quantum affine algebra. Haglund, Haiman, Loehr, Remmel, and Ulyanov \cite{combdiag} redefined them in terms of tuples of skew shapes in their study of diagonal coinvariants. Haglund, Haiman, and Loehr \cite{combmacdonald} found a combinatorial formula for Macdonald polynomials, which implies a positive expansion in terms of these LLT polynomials $G_{\bm\lambda}(\bm x;q)$. LLT polynomials are also closely connected to chromatic quasisymmetric functions and to the Frobenius series of the space of diagonal harmonics \cite{shuffle}. Grojnowski and Haiman \cite{lltpos} proved that LLT polynomials, and therefore Macdonald polynomials, are Schur-positive using Kazhdan--Lusztig theory, but it remains a major open problem to find an explicit combinatorial Schur-positive expansion. We give a brief account of some recent results in this direction.\\

In the unicellular case, meaning that every skew shape of $\bm\lambda$ consists of a single cell, we can associate a unit interval graph to $\bm\lambda$. Huh, Nam, and Yoo \cite{lltunicellschur} found an explicit Schur-positive expansion whenever this graph is a melting lollipop, namely \begin{equation}G_{\bm\lambda}(\bm x;q)=\sum_{T\in\text{SYT}_n}q^{\text{wt}_{\bm a}(T)}s_{\text{shape}(T)}.\end{equation} Moreover, they proved that for arbitrary unit interval graphs, this formula gives the correct coefficient of $s_\mu$ whenever the partition $\mu$ is a hook. \\

More generally, we focus on the horizontal-strip case, meaning that every skew shape of $\bm\lambda$ is a row. Grojnowski and Haiman \cite{lltpos} showed that if the rows of $\bm\lambda$ are nested, then $G_{\bm\lambda}(\bm x;q)$ is a transformed modified Hall--Littlewood polynomial and so its Schur expansion is given by the celebrated Lascoux--Sch\"{u}tzenberger cocharge formula \cite{charge}, namely \begin{equation}G_{\bm\lambda}(\bm x;q)=\tilde H_\lambda(\bm x;q)=\sum_{T\in\text{SSYT}(\lambda)}q^{\text{cocharge}(T)}s_{\text{shape}(T)}.\end{equation}

Alexandersson and Uhlin \cite{lltmpath} found a generalization of cocharge to prove an analogous formula when the rows of $\bm\lambda$ come from a skew shape $\sigma/\tau$ with no column having more than two cells. They formulated it for vertical-strips but we can equivalently state it as \begin{equation}G_{\bm\lambda}(\bm x;q)=\sum_{T\in\text{SSYT}(\alpha)}q^{\text{cocharge}_\tau(T)}s_{\text{shape}(T)}.\end{equation}

D'Adderio \cite{lltepos} used recurrences in terms of Schr\"{o}der paths to prove that the shifted vertical-strip LLT polynomial $G_{\bm\lambda}(\bm x;q+1)$ is a positive linear combination of elementary symmetric functions. Alexandersson conjectured \cite{lltunicellepos} and then proved with Sulzgruber \cite{lltcombe} the explicit combinatorial formula \begin{equation}G_{\bm\lambda}(\bm x;q+1)=\sum_{\theta\in\mathcal O(P)}q^{\text{asc}(\theta)}e_{\lambda(\theta)}\end{equation} in terms of acyclic orientations of a decorated unit interval graph associated to $\bm\lambda$. \\

In this paper, we define a \emph{weighted graph} $\Pi$ associated to $\bm\lambda$. In Section \ref{sec:wgraph}, we use our weighted graph to express linear recurrences of horizontal-strip LLT polynomials. We further generalize cocharge and apply our recurrences in Section \ref{sec:caterpillar} to prove the explicit combinatorial Schur-positive formula \begin{equation}G_{\bm\lambda}(\bm x;q)=\sum_{T\in\text{SSYT}(\alpha)}q^{\text{cocharge}_\Pi(T)}s_{\text{shape}(T)}\end{equation} whenever the weighted graph $\Pi$ is triangle-free. We also prove that the largest power of $q$ in the LLT polynomial $G_{\bm\lambda}(\bm x;q)$ is the total edge weight of $\Pi$. 

\section{Background}\label{sec:background}

A \emph{partition} $\sigma$ is a finite sequence of nonincreasing positive integers $\sigma=\sigma_1\cdots\sigma_\ell$. By convention, we set $\sigma_i=0$ if $i>\ell$. A \emph{skew diagram} $\lambda$ is a subset of $\mathbb Z\times\mathbb Z$ of the form \begin{equation}\lambda=\sigma/\tau=\{(i,j): \ i\geq 1, \ \tau_i+1\leq j\leq \sigma_i\}\end{equation} for some partitions $\sigma$ and $\tau$ with $\sigma_i\geq\tau_i$ for every $i$. When $\tau$ is empty, we write $\sigma$ instead of $\sigma/\emptyset$. The elements of $\lambda$ are called \emph{cells} and the \emph{content} of a cell $u=(i,j)\in\lambda$ is the integer $c(u)=j-i$. We will focus heavily on \emph{rows}, which are skew diagrams of the form \begin{equation}R=a/b=\{(1,j): \ b+1\leq j\leq a\}\end{equation} for some $a\geq b\geq 0$. We denote by $c(R)=\{b,b+1,\ldots,a-1\}$ the set of contents of cells in $R$ and by $\ell(R)=b$ and $r(R)=a-1$ the smallest and largest contents of $c(R)$ respectively. A \emph{semistandard Young tableau (SSYT)} of shape $\lambda$ is a function $T:\lambda\to\{1,2,3,\ldots\}$ that satisfies \begin{equation}T_{i,j}\leq T_{i,j+1}\text{ and }T_{i,j}<T_{i+1,j},\end{equation} where we write $T_{i,j}$ to mean $T((i,j))$. The \emph{weight} of $T$ is the sequence $w(T)=(w_1,w_2,\ldots)$, where $w_i=|T^{-1}(i)|$ is the number of times the integer $i$ appears. We denote by $\text{SSYT}_\lambda$ the set of SSYT of shape $\lambda$ and by $\text{SSYT}(\alpha)$ the set of SSYT of weight $\alpha$. We define the \emph{skew Schur function} of shape $\lambda=\sigma/\tau$ to be \begin{equation}s_\lambda=\sum_{T\in\text{SSYT}_\lambda}\bm x^T,\end{equation} where $\bm x^T$ is the monomial $x_1^{w_1}x_2^{w_2}\cdots$. When $\tau$ is empty, we call $s_\lambda$ a \emph{Schur function}. \\

A \emph{multiskew partition} is a finite sequence of skew diagrams $\bm\lambda=(\lambda^{(1)},\ldots,\lambda^{(n)})$. If each $\lambda^{(i)}$ is a row, then we call $\bm\lambda$ a \emph{horizontal-strip}. We denote by \begin{equation}\text{SSYT}_{\bm\lambda}=\{\bm T=(T^{(1)},\ldots,T^{(n)}): \ T^{(i)}\in\text{SSYT}_{\lambda^{(i)}}\}\end{equation} the set of \emph{semistandard multiskew tableaux} of shape $\bm\lambda$. Cells $u\in\lambda^{(i)}$ and $v\in\lambda^{(j)}$ with $i<j$ \emph{attack} each other if $c(u)=c(v)$ or $c(u)=c(v)+1$. The skew shapes $\lambda^{(i)}$ and $\lambda^{(j)}$ \emph{attack} each other if some cells $u\in\lambda^{(i)}$ and $v\in\lambda^{(j)}$ attack each other. Entries $T^{(i)}(u)$ and $T^{(j)}(v)$ with $i<j$ form an \emph{inversion} if either \begin{itemize} 
\item $c(u)=c(v)$ and $T^{(i)}(u)>T^{(j)}(v)$, or 
\item $c(u)=c(v)+1$ and $T^{(j)}(v)>T^{(i)}(u)$. 
\end{itemize}
We denote by $\text{inv}(\bm T)$ the number of inversions of $\bm T$. Now we define the \emph{LLT polynomial} \begin{equation}\label{eq:llt} G_{\bm\lambda}(\bm x;q)=\sum_{\bm T\in\text{SSYT}_{\bm\lambda}}q^{\text{inv}(\bm T)}\bm x^{\bm T}.\end{equation}
\begin{example} \label{ex:llt} Let $\bm\lambda=(4/0,5/2,2/0)$. When $\bm\lambda$ is a horizontal-strip we draw it so that cells of the same content are aligned vertically as on the left. We have written the content in each cell using our convention that content increases from left to right. We have also drawn two tableaux $\bm T,\bm U\in\text{SSYT}_{\bm\lambda}$ with dotted red lines indicating the inversions.
$$
\begin{tikzpicture}
\draw (-1.25,0.75) node (0) {$\bm \lambda=$};
\draw (0.25,-0.25) node (1) {0} (0.25,1.75) node (2) {0} (0.75,-0.25) node (3) {1} (0.75,1.75) node (4) {1} (1.25,-0.25) node (5) {2} (1.25,0.75) node (6) {2} (1.75,-0.25) node (7) {3} (1.75,0.75) node (8) {3} (2.25,0.75) node (9) {4};
\draw (0,0) -- (0.5,0) -- (0.5,-0.5) -- (0,-0.5) -- (0,0) (0,1.5) -- (0.5,1.5) -- (0.5,2) -- (0,2) -- (0,1.5) (0.5,-0.5) -- (0.5,0) -- (1,0) -- (1,-0.5) -- (0.5,-0.5) (0.5,1.5) -- (0.5,2) -- (1,2) -- (1,1.5) -- (0.5,1.5) (1,0) -- (1.5,0) -- (1.5,-0.5) -- (1,-0.5) -- (1,0) (1,0.5) -- (1,1) -- (1.5,1) -- (1.5,0.5) -- (1,0.5) (1.5,0) -- (1.5,-0.5) -- (2,-0.5) -- (2,0) -- (1.5,0) (1.5,0.5) -- (1.5,1) -- (2,1) -- (2,0.5) -- (1.5,0.5) (2,0.5) -- (2,1) -- (2.5,1) -- (2.5,0.5) -- (2,0.5);
\end{tikzpicture}\hspace{50pt}
\begin{tikzpicture}
\draw (-1.25,0.75) node (0) {$\bm T=$};
\draw (0.25,-0.25) node (1) {3} (0.25,1.75) node (2) {2} (0.75,-0.25) node (3) {3} (0.75,1.75) node (4) {2} (1.25,-0.25) node (5) {3} (1.25,0.75) node (6) {1} (1.75,-0.25) node (7) {3} (1.75,0.75) node (8) {1} (2.25,0.75) node (9) {1};
\draw (0,0) -- (0.5,0) -- (0.5,-0.5) -- (0,-0.5) -- (0,0) (0,1.5) -- (0.5,1.5) -- (0.5,2) -- (0,2) -- (0,1.5) (0.5,-0.5) -- (0.5,0) -- (1,0) -- (1,-0.5) -- (0.5,-0.5) (0.5,1.5) -- (0.5,2) -- (1,2) -- (1,1.5) -- (0.5,1.5) (1,0) -- (1.5,0) -- (1.5,-0.5) -- (1,-0.5) -- (1,0) (1,0.5) -- (1,1) -- (1.5,1) -- (1.5,0.5) -- (1,0.5) (1.5,0) -- (1.5,-0.5) -- (2,-0.5) -- (2,0) -- (1.5,0) (1.5,0.5) -- (1.5,1) -- (2,1) -- (2,0.5) -- (1.5,0.5) (2,0.5) -- (2,1) -- (2.5,1) -- (2.5,0.5) -- (2,0.5);
\draw [color=red, dashed] (0.25,-0.25) -- (0.25,1.75);
\draw [color=red, dashed] (0.75,-0.25) -- (0.75,1.75);
\draw [color=red, dashed] (1.25,-0.25) -- (1.25,0.75);
\draw [color=red, dashed] (1.75,-0.25) -- (1.75,0.75);
\draw [color=red, dashed] (0.75,1.75) -- (1.25,0.75);
\end{tikzpicture}\hspace{50pt}\begin{tikzpicture}
\draw (-1.25,0.75) node (0) {$\bm U=$} (0.25,-0.25) node (1) {1}
 (0.25,1.75) node (2) {2} (0.75,-0.25) node (3) {1}
 (0.75,1.75) node (4) {3} (1.25,-0.25) node (5) {2}
 (1.25,0.75) node (6) {3} (1.75,-0.25) node (7) {4}
 (1.75,0.75) node (8) {3} (2.25,0.75) node (9) {4};
\draw (0,0) -- (0.5,0) -- (0.5,-0.5) -- (0,-0.5) -- (0,0) (0,1.5) -- (0.5,1.5) -- (0.5,2) -- (0,2) -- (0,1.5) (0.5,-0.5) -- (0.5,0) -- (1,0) -- (1,-0.5) -- (0.5,-0.5) (0.5,1.5) -- (0.5,2) -- (1,2) -- (1,1.5) -- (0.5,1.5) (1,0) -- (1.5,0) -- (1.5,-0.5) -- (1,-0.5) -- (1,0) (1,0.5) -- (1,1) -- (1.5,1) -- (1.5,0.5) -- (1,0.5) (1.5,0) -- (1.5,-0.5) -- (2,-0.5) -- (2,0) -- (1.5,0) (1.5,0.5) -- (1.5,1) -- (2,1) -- (2,0.5) -- (1.5,0.5) (2,0.5) -- (2,1) -- (2.5,1) -- (2.5,0.5) -- (2,0.5);
\draw [color=red, dashed] (0.25,1.75) -- (0.75,-0.25);
\draw [color=red, dashed] (0.75,1.75) -- (1.25,-0.25);
\draw [color=red, dashed] (1.75,-0.25) -- (1.75,0.75);
\end{tikzpicture} $$
The tableau $\bm T$ contributes $q^5x_1^3x_2^2x_3^4$ to \eqref{eq:llt} and the tableau $\bm U$ contributes $q^3x_1^2x_2^2x_3^3x_4^2$. We can expand the LLT polynomial $G_{\bm\lambda}(\bm x;q)$ in the basis of Schur functions as \begin{align}G_{\bm\lambda}(\bm x;q)&=q^5s_{432}+q^5s_{441}+q^5s_{522}+(q^5+q^4)s_{531}+2q^4s_{54}+2q^4s_{621}\\\nonumber&+(q^4+2q^3)s_{63}+q^3s_{711}+(2q^3+q^2)s_{72}+(q^2+q)s_{81}+s_9.\end{align}
\end{example} 
We cite some helpful properties of LLT polynomials. The first three are immediate from the definition. 
\begin{proposition}\label{prop:lltfacts}Let $\bm\lambda=(\lambda^{(1)},\ldots,\lambda^{(n)})$ be a multiskew partition.
\begin{enumerate}
\item When $q=1$ the LLT polynomial $G_{\bm\lambda}(\bm x;1)$ is a product of Schur functions $\Pi_{i=1}^n s_{\lambda^{(i)}}$. 
\item Suppose that $\bm\lambda$ can be partitioned into two subsequences of skew partitions $\bm\mu=(\mu^{(i_1)},\ldots,\mu^{(i_r)})$ and $\bm\nu=(\nu^{(j_1)},\ldots,\nu^{(j_{r'})})$ such that no $\mu^{(i_t)}$ attacks any $\nu^{(j_{t'})}$. Then $G_{\bm\lambda}(\bm x;q)=G_{\bm\mu}(\bm x;q)G_{\bm\nu}(\bm x;q)$. 
\item Let $\kappa(\bm\lambda)=(\lambda^{(n)+},\lambda^{(1)},\ldots,\lambda^{(n-1)})$, where $\lambda^{(n)+}=\{(i,j+1): \ (i,j)\in\lambda^{(n)}\}$. Then $G_{\kappa(\bm\lambda)}(\bm x;q)=G_{\bm\lambda}(\bm x;q)$. 
\item \cite[Theorem 6.1]{lltoriginal} The LLT polynomial $G_{\bm\lambda}(\bm x;q)$ is a symmetric function.
\item \cite[Lemma 10.1]{combmacdonald} Let $\omega$ denote the standard involution on the algebra of symmetric functions. Denote by $\omega(\bm\lambda)$ the multiskew partition where each skew shape of $\bm\lambda$ is conjugated and the tuple is reversed. Then \begin{equation}G_{\omega(\bm\lambda)}(\bm x;q)=q^{I(\bm\lambda)}\omega G_{\bm\lambda}(\bm x;q^{-1}),\end{equation} where $I(\bm\lambda)$ is the number of attacking pairs of cells in $\bm\lambda$. In particular, results about horizontal-strip and vertical-strip LLT polynomials are equivalent. In this paper, we choose to consider horizontal-strips. 
\end{enumerate}
\end{proposition}

In Example \ref{ex:llt}, we saw that $G_{(4/0,5/2,2/0)}(\bm x;q)$ is \emph{Schur-positive}, meaning that it is an $\mathbb N[q]$-linear combination of Schur functions. In fact, this property holds in general. 

\begin{theorem} \label{thm:lltpos}\cite[Corollary 6.9]{lltpos} For any multiskew partition $\bm\lambda$, the LLT polynomial $G_{\bm\lambda}(\bm x;q)$ is Schur-positive. \end{theorem}

The special case where $\bm\lambda$ is a horizontal-strip was proven in \cite[Theorem 3.1.3]{combdiag} using some results introduced in \cite{lrkl}. Both this special case and Theorem \ref{thm:lltpos} were proven using Kazhdan--Lusztig theory. It is a major open problem to find an explicit combinatorial Schur-positive expansion of LLT polynomials. We conclude this section with a discussion of a successful solution in a special case. We first introduce the jeu de taquin algorithm.

\begin{definition}
Let $\lambda=\sigma/\tau$ be a skew shape and $T\in\text{SSYT}_{\lambda}$. An \emph{inside corner} of $\lambda$ is a cell $u\in\tau$ such that $\lambda\cup\{u\}$ is a skew shape. The \emph{jeu de taquin slide} of $T$ into an inside corner $u$ is obtained as follows. There is a cell $v$ directly above or directly right of $u$. If both, let $v$ be the one with a smaller entry, and if they have the same entry, let $v$ be the cell above $u$. We move the entry in $v$ to $u$. We continue by considering the cells directly above and directly to the right of $v$ and we stop when we vacate a cell on the outer boundary of $\lambda$, so that the result is a skew tableau. The \emph{rectification} of $T$ is the tableau obtained by successive jeu de taquin slides until the result is of partition shape. \end{definition}

\begin{theorem} \label{thm:jdt} \cite[Theorem A1.2.4]{enum2} The rectification of $T$ does not depend on the sequence of choices of inside corners into which the jeu de taquin slides are performed. \end{theorem} 

The following definition of cocharge is not the classical one but it is an equivalent characterization and it has the advantage that the weight of $T$ is not required to be a partition. 

\begin{definition} Let $T\in\text{SSYT}_{\lambda}$. The \emph{cocharge} of $T$ is the integer defined by the following three properties.
\begin{enumerate}
\item Cocharge is invariant under jeu de taquin slides.
\item If $T$ is a single row, then $\text{cocharge}(T)=0$. 
\item Suppose that the shape of $T$ is disconnected so that $T=X\cup Y$ with every entry of $X$ above and left of every entry of $Y$. Let $i$ be the smallest entry of $T$ and suppose that no entry of $X$ is equal to $i$. Let $S$ be a tableau obtained by swapping $X$ and $Y$ so that every entry of $Y$ is above and left of every entry of $X$. Then $\text{cocharge}(T)=\text{cocharge}(S)+|X|$. 
\end{enumerate}
\end{definition}

The classical definition of cocharge satisfies these three properties \cite[Lemma 6.6.6]{plactic}. Conversely, these properties suffice to calculate $\text{cocharge}(T)$ by using the following process called \emph{catabolism}, which was introduced in \cite[Problem 6.6.1]{plactic}. If $T$ is a single row, then $\text{cocharge}(T)=0$, otherwise by applying jeu de taquin slides, we can slide the top row of $T$ to the left to disconnect it, swap the pieces, and rectify to produce a new tableau $S$ of smaller cocharge. Repeated catabolism will terminate with a single row of cocharge zero.

\begin{example} Given the tableau $T$ below left, repeated catabolism produces the following sequence of tableaux and $\text{cocharge}(T)=1+2+1=4$. $$\tableau{3\\2&3\\1&2&2}\hspace{30pt}\tableau{\\2&3\\1&2&2&3}\hspace{30pt}\tableau{\\3\\1&2&2&2&3}\hspace{30pt}\tableau{\\ \\ 1&2&2&2&3&3}$$\end{example}
\begin{example} \label{ex:cochij}
If $T$ is a tableau of partition shape and every entry equal to $i$ or $j$ with $i<j$, then $T$ must be of the form $$T=\tableau{j&\cdot&\cdot&\cdot&j\\i&\cdot&\cdot&\cdot&i&i&\cdot&\cdot&\cdot&i&j&\cdot&\cdot&\cdot&j}$$ and a single catabolism will produce a tableau with one row, so the cocharge of $T$ is the number of entries in the second row. We can think of $\text{cocharge}(T)$ as measuring the extent to which there are entries above others.
\end{example}

We now present a case in which a combinatorial formula for the Schur-positivity of a horizontal-strip LLT polynomial is known.

\begin{theorem} \label{thm:hl} \cite[Theorem 7.15]{lltpos} Let $\bm\lambda=(R_1,\ldots,R_n)$ be a horizontal-strip such that $\ell(R_1)\leq\cdots\leq\ell(R_n)$ and $r(R_1)\geq\cdots\geq r(R_n)$ and let $\lambda_i=|R_i|$. Then the LLT polynomial $G_{\bm\lambda}(\bm x;q)$ is a transformed modified Hall--Littlewood polynomial, whose Schur expansion is known \cite{charge} to be \begin{equation}G_{\bm\lambda}(\bm x;q)=\tilde H_\lambda(\bm x;q)=\sum_{T\in\text{SSYT}(\lambda)}q^{\text{cocharge}(T)}s_{\text{shape}(T)}.\end{equation}
\end{theorem}

The central problem of this paper is to generalize cocharge in order to prove an analogous combinatorial formula for the Schur expansion of any horizontal-strip LLT polynomial. Our main result, Theorem \ref{thm:caterpillar}, is such a combinatorial formula in the case where no three rows of $\bm\lambda$ pairwise attack each other. Our strategy is to define a weighted graph $\Pi(\bm\lambda)$ associated to $\bm\lambda$ and to use it to express linear recurrences of LLT polynomials.

\section{A weighted graph description of horizontal-strip LLT polynomials}\label{sec:wgraph}

We begin by defining our weighted graph $\Pi(\bm\lambda)$. 

\begin{definition} \label{def:wgraph} Let $R$ and $R'$ be rows. We define the integer \begin{equation}\label{eq:mrirj} M(R,R')=\begin{cases} |c(R)\cap c(R')| & \text{ if }\ell(R)\leq\ell(R'),\\|c(R)\cap c(R'^+)| & \text{ if }\ell(R)>\ell(R'),\end{cases}\end{equation} where as before, $R'^+=\{(1,j+1): \ (1,j)\in R'\}$. Note that $0\leq M(R,R')\leq\min\{|R|,|R'|\}$. \end{definition}

\begin{definition} Let $\bm\lambda=(R_1,\ldots,R_n)$ be a horizontal-strip. Consider the cells of $\bm\lambda$ that are the rightmost cells in their row and label these cells $1,\ldots,n$ in \emph{content reading order}, meaning in order of increasing content and from bottom to top along constant content lines. We define the \emph{weighted graph} $\Pi(\bm\lambda)$ with vertices $v_1,\ldots,v_n$ as follows. The weight of a vertex $v_i$, denoted $|v_i|$, is the size of the row $R_{i'}$ whose rightmost cell is labelled $i$. Vertices $v_i$ and $v_j$ corresponding to rows $R_{i'}$ and $R_{j'}$ with $i'<j'$ are joined by an edge if $R_{i'}$ and $R_{j'}$ attack each other and the weight of the edge $(v_i,v_j)$, denoted $M_{i,j}$, is given by $M(R_{i'},R_{j'})$.
\end{definition}

\begin{example} \label{ex:wgraph} Let $\bm\lambda=(R_1,R_2,R_3)=(4/0,5/2,2/0)$ as in Example \ref{ex:llt}. We have drawn $\bm\lambda$ with the rightmost cells in each row labelled in content reading order, and we have drawn $\Pi(\bm\lambda)$ below right. We have $M(R_1,R_2)=2$, $M(R_1,R_3)=2$, and $M(R_2,R_3)=1$.

$$
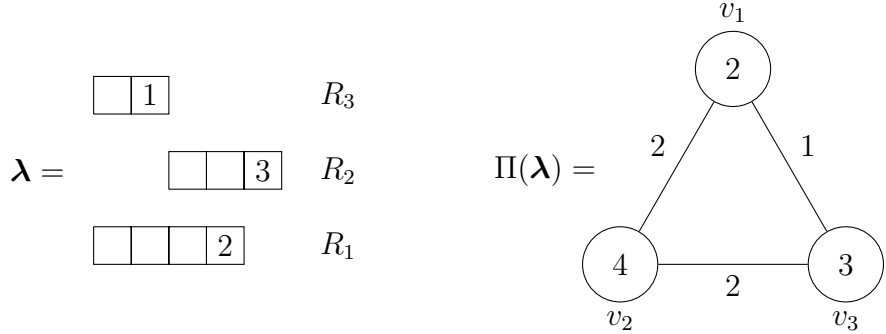
\begin{tikzpicture}
\draw (-0.75,0.75) node (0) {$\bm \lambda=$};
\draw (0,0) -- (0.5,0) -- (0.5,-0.5) -- (0,-0.5) -- (0,0) (0,1.5) -- (0.5,1.5) -- (0.5,2) -- (0,2) -- (0,1.5) (0.5,-0.5) -- (0.5,0) -- (1,0) -- (1,-0.5) -- (0.5,-0.5) (0.5,1.5) -- (0.5,2) -- (1,2) -- (1,1.5) -- (0.5,1.5) (1,0) -- (1.5,0) -- (1.5,-0.5) -- (1,-0.5) -- (1,0) (1,0.5) -- (1,1) -- (1.5,1) -- (1.5,0.5) -- (1,0.5) (1.5,0) -- (1.5,-0.5) -- (2,-0.5) -- (2,0) -- (1.5,0) (1.5,0.5) -- (1.5,1) -- (2,1) -- (2,0.5) -- (1.5,0.5) (2,0.5) -- (2,1) -- (2.5,1) -- (2.5,0.5) -- (2,0.5);
\draw (0.75,1.75) node (4) {1} (1.75,-0.25) node (7) {2} (2.25,0.75) node (9) {3} (3.25,-0.25) node (a) {$R_1$} (3.25,0.75) node (b) {$R_2$} (3.25,1.75) node (c) {$R_3$} (6,0.75) node (p) {$\Pi(\bm\lambda)=$};
\node[shape=circle,draw=black,minimum size=10mm](1) at (7,-0.5) {4};
\node at (7,-1.25) {$v_2$};
\node[shape=circle,draw=black,minimum size=10mm](2) at (10,-0.5) {3};
\node at (10,-1.25) {$v_3$};
\node[shape=circle,draw=black,minimum size=10mm](3) at (8.5,2.1) {2};
\node at (8.5,2.85) {$v_1$};
\path [-](1) edge node [below]{$2$} (2);
\path [-](1) edge node [above left]{$2$} (3);
\path [-](2) edge node [above right]{$1$} (3);
\end{tikzpicture}
$$\end{example}

\begin{remark} It follows immediately from the definition that $M(R,R')=M(R'^+,R)$ and therefore the map $\kappa$ from Proposition \ref{prop:lltfacts}, Part 3 preserves our weighted graph. We can think of the integers $M(R_i,R_j)$ as measuring the extent to which the rows $R_i$ and $R_j$ attack each other. We now present a result to make this idea precise. \end{remark}

\begin{theorem} \label{thm:maxinv} Let $\bm\lambda=(\lambda^{(1)},\ldots,\lambda^{(n)})$ be a multiskew partition. Define the integer \begin{equation}M(\bm\lambda)=\sum_{1\leq i<j\leq n}\sum_{\substack{R\text{ a row of }\lambda^{(i)}\\R'\text{ a row of }\lambda^{(j)}}}M(R,R').\end{equation} Then every $\bm T\in\text{SSYT}_{\bm\lambda}$ has $\text{inv}(\bm T)\leq M(\bm\lambda)$. Moreover, this maximum is attained. \end{theorem} 

\begin{remark} Theorem \ref{thm:maxinv} tells us that $M(\bm\lambda)$ is the largest power of $q$ in the LLT polynomial $G_{\bm\lambda}(\bm x;q)$. In particular, if $G_{\bm\lambda}(\bm x;q)=G_{\bm\mu}(\bm x;q)$, then $M(\bm\lambda)=M(\bm\mu)$. We also remark that by Proposition \ref{prop:lltfacts}, Part 5, the smallest power of $q$ is $I(\bm\lambda)-M(\omega(\bm\lambda))$. If $\bm\lambda$ is a horizontal-strip, the largest power of $q$ in the LLT polynomial is the total edge weight of $\Pi(\bm\lambda)$. In the Hall--Littlewood case of Theorem \ref{thm:hl}, we have $M(\bm\lambda)=n(\lambda)=\sum_i(i-1)\lambda_i$. \end{remark}

\begin{proof} We first observe that an entry $x\in\bm T$ makes at most one inversion with the entries in each row. If $x$ made an inversion with $y\in \bm T$ and $z\in\bm T$ in the same row with $y$ to the left of $z$, then we would require $y>x>z$, but $y\leq z$ because the rows of $\bm T$ are weakly increasing. \\

To prove the first statement, we show that for rows $R$ of $\lambda^{(i)}$ and $R'$ of $\lambda^{(j)}$ with $i<j$, the number of inversions $\text{inv}_{R,R'}(\bm T)$ of $\bm T$ between cells in these rows is at most $M(R,R')$. Suppose that $\ell(R)\leq\ell(R')$ and let $v_1$ be the leftmost cell of row $R'$. If there is no cell $u_1$ of $R$ with $c(u_1)=c(v_1)$, then $\text{inv}_{R,R'}(\bm T)=0$, so suppose otherwise. Let $v_2,\ldots,v_m$ be the cells in $R'$ to the right of $v_1$ and let $u_2,\ldots,u_{m'}$ be the cells in $R$ to the right of $u_1$ as pictured below left. By our observation above, we have $\text{inv}_{R,R'}(\bm T)\leq\min\{m,m'\}=M(R,R')$. Similarly, if $\ell(R)>\ell(R')$, let $u_1$ be the leftmost cell of row $R$. If there is no cell $v_1$ of $R'$ with $c(v_1)=c(u_1)-1$, then $\text{inv}_{R,R'}(\bm T)=0$, so suppose otherwise. Let $v_2,\ldots,v_m$ be the cells in $R'$ to the right of $v_1$ and let $u_2,\ldots,u_{m'}$ be the cells in $R$ to the right of $u_1$ as pictured below right. Again, by our observation above, we have $\text{inv}_{R,R'}(\bm T)\leq\min\{m,m'\}=M(R,R')$. 

$$\tableau{&&&v_1&v_2&\cdot&\cdot&\cdot&\cdot&\cdot&\cdot&\cdot&v_m\\ \\ \cdot & \cdot & \cdot & u_1&u_2&\cdot&\cdot&\cdot&u_{m'}}\hspace{50pt}\tableau{\cdot&\cdot&\cdot&v_1&v_2&\cdot&\cdot&\cdot&v_m\\ \\ &&&&u_1&u_2&\cdot & \cdot & \cdot & \cdot&\cdot&\cdot&u_{m'}}$$

To prove the second statement, we construct a tableau $\bm T\in\text{SSYT}_{\bm\lambda}$ with $\text{inv}(\bm T)=M(\bm\lambda)$. Consider the cells of $\bm\lambda$ that are the leftmost cells in their row. We label these cells $1,2,3,\ldots$ in reverse content reading order, meaning in order of decreasing content and from top to bottom along constant content lines. We then fill each cell with the label of the leftmost cell in its row. As an illustration, when $\bm\lambda=(4/0,5/2,2/0)$, the tableau we have constructed is the tableau $\bm T$ in Example \ref{ex:llt}.\\

By construction, the rows of $\bm T$ are constant and therefore weakly increasing, and if a cell $u$ is directly below a cell $v$ in the same skew diagram, then the leftmost cell in the row of $u$ must precede that of $v$ in reverse content reading order, so $\bm T(u)<\bm T(v)$ and indeed $\bm T\in\text{SSYT}_{\bm\lambda}$. Finally, we show that for rows $R$ of $\lambda^{(i)}$ and $R'$ of $\lambda^{(j)}$ with $i<j$, we have $\text{inv}_{R,R'}(\bm T)=M(R,R')$. If $\ell(R)\leq\ell(R')$, then every cell in $R$ has a strictly larger entry than every cell in $R'$, so we have an inversion $\bm T(u)>\bm T(v)$ for every $u\in R$ and $v\in R'$ with $c(u)=c(v)$, meaning $|c(R)\cap c(R')|=M(R,R')$ inversions in all. Similarly, if $\ell(R)>\ell(R')$, then every cell in $R'$ has a strictly larger entry than every cell in $R$, so we have an inversion $\bm T(v)>\bm T(u)$ for every $u\in R$ and $v\in R'$ with $c(u)=c(v)+1$, meaning $|c(R)\cap c(R'^+)|=M(R,R')$ inversions in all. \end{proof}

We continue to use the integers $M(R,R')$ to describe relationships between rows. The following definition will be justified by Lemma \ref{lem:commuting}. 

\begin{definition} We say that two rows $R$ and $R'$ \emph{commute} if $M(R,R')=M(R',R)$. We write $R\leftrightarrow R'$ if $R$ and $R'$ commute and $R\nleftrightarrow R'$ otherwise. \end{definition}

\begin{proposition} \label{prop:mrirj} Let $R$ and $R'$ be rows with $\ell(R)\leq\ell(R')$. 
\begin{enumerate}
\item If $r(R)<\ell(R')-1$, then \begin{equation}M(R,R')=M(R',R)=0,\text{ so }R\leftrightarrow R'.\end{equation}

\item If $\ell(R')=\ell(R)$ or $r(R')\leq r(R)$, then \begin{equation}M(R,R')=M(R',R)=\min\{|R|,|R'|\},\text{ so }R\leftrightarrow R'.\end{equation}

\item Otherwise, we have $\ell(R)<\ell(R')\leq r(R)+1\leq r(R')$ and \begin{equation}M(R,R')=r(R)-\ell(R')+1\text{ and }M(R',R)=r(R)-\ell(R')+2,\text{ so }R\nleftrightarrow R'.\end{equation}
\end{enumerate}
In particular, if $R\leftrightarrow R'$, then $M(R,R')$ is either $0$ or $\min\{|R|,|R'|\}$. 
\end{proposition}

\begin{proof} We compute directly from the definition. If $r(R)<\ell(R')-1$, then $R\cap R'=R'\cap R^+=\emptyset$. If $\ell(R')=\ell(R)$, then either $R\subseteq R'$ or $R'\subseteq R$, and if $\ell(R)<\ell(R')$ and $r(R')\leq r(R)$, then $R'\subseteq R$ and $R'\subseteq R^+$. Otherwise, the only remaining case is when $\ell(R)\leq\ell(R')-1\leq r(R)\leq r(R')-1$, in which case we have
\begin{align}M(R,R')&=|c(R)\cap c(R')|=|\{\ell(R'),\ldots,r(R)\}|=r(R)-\ell(R')+1\text{ and }\\
M(R',R)&=|c(R')\cap c(R^+)|=|\{\ell(R'),\ldots,r(R)+1\}|=r(R)-\ell(R')+2.\end{align}
This completes the proof.
\end{proof}

\begin{remark} For a more visual description, we have that two rows $R$ and $R'$ commute if and only if they are disjoint and separated by at least one cell, or if one is contained in the other. In the examples below, only the pairs of rows on the left and on the right commute.$$\tableau{&&&&&& \ & \ & \ \\ \\ \ & \ & \ & \ & \ }\hspace{60pt}\tableau{&&& \ & \ & \ \\ \\ \ & \ & \ & \ & \ }\hspace{100pt}\tableau{& \ & \ & \ \\ \\ \ & \ & \ & \ & \ }$$
\end{remark}

We now introduce the language of LLT-equivalence to describe some local linear relations of LLT polynomials. If $\bm\lambda=(\lambda^{(1)},\ldots,\lambda^{(n)})$ and $\bm\mu=(\mu^{(1)},\ldots,\mu^{(n')})$ are multiskew partitions, then we denote by $\bm\lambda\cdot\bm\mu$ the concatenation $(\lambda^{(1)},\ldots,\lambda^{(n)},\mu^{(1)},\ldots,\mu^{(n')})$. 

\begin{definition} Multiskew partitions $\bm\lambda$ and $\bm\mu$ are \emph{LLT-equivalent}, denoted $\bm\lambda\cong\bm\mu$, if for every multiskew partition $\bm\nu$ we have the equality of LLT polynomials \begin{equation}G_{\bm\lambda\cdot\bm\nu}(\bm x;q)=G_{\bm\mu\cdot\bm\nu}(\bm x;q).\end{equation}
More generally, finite formal $\mathbb Q(q)$-combinations of multiskew partitions $\sum_i a_i(q)\bm\lambda_i$ and $\sum_j b_j(q)\bm\mu_j$ are \emph{LLT-equivalent}, denoted $\sum_i a_i(q)\bm\lambda_i\cong\sum_jb_j(q)\bm\mu_j$, if for every $\bm\nu$ we have \begin{equation}\sum_i a_i(q)G_{\bm\lambda_i\cdot\bm\nu}(\bm x;q)=\sum_jb_j(q)G_{\bm\mu_j\cdot\bm\nu}(\bm x;q).\end{equation}
\end{definition}
\begin{remark} By repeatedly applying the map $\kappa$ from Proposition \ref{prop:lltfacts}, Part 3, we have that if $\bm\lambda\cong\bm\mu$, then $G_{\bm\nu\cdot\bm\lambda\cdot\bm\nu'}(\bm x;q)=G_{\bm\nu\cdot\bm\mu\cdot\bm\nu'}(\bm x;q)$ for every multiskew partitions $\bm\nu$ and $\bm\nu'$. We can think of LLT-equivalence as a \emph{local} linear relation because we can locally replace $\bm\lambda$ with $\bm\mu$ while preserving the LLT polynomial. \end{remark}

We can prove an LLT-equivalence by rearranging to an LLT-equivalence of $\mathbb N[q]$-linear combinations of multiskew partitions and finding an appropriate bijection of tableaux. 

\begin{theorem} \label{thm:llteq} \cite[Theorem 2.2.1]{lltkschurband} Let $\sum_i q^{a_i}\bm\lambda_i$ and $\sum_j q^{b_j}\bm\mu_j$ be $\mathbb N[q]$-linear combinations of multiskew partitions. Then they are LLT-equivalent if there exists a bijection \begin{equation}f:\bigsqcup_i \text{SSYT}_{\bm\lambda_i}\to\bigsqcup_j\text{SSYT}_{\bm\mu_j}\end{equation} such that if $f$ maps $\bm T\in\text{SSYT}_{\bm\lambda_i}$ to $\bm U\in\text{SSYT}_{\bm\mu_j}$, then \begin{equation}\label{eq:inversionobeying} \text{inv}(\bm T)+a_i=\text{inv}(\bm U)+b_j,\end{equation} and for every $c\in\mathbb Z$, the multiset of entries in cells of content $c$ are preserved, that is \begin{equation}\label{eq:contentpreserving} \{\bm T(u): \ c(u)=c\}=\{\bm U(u): \ c(u)=c\}.\end{equation} 
\end{theorem}

We now use Theorem \ref{thm:llteq} to establish some valuable LLT-equivalence relations. These relations appear in \cite[Lemma 5.2]{lltepos} and \cite[Theorem 2.1]{lltcombe} in terms of operators on Dyck paths and Schr\"{o}der paths.

\begin{lemma} \label{lem:llteq} Let $R$ and $R'$ be rows such that $\ell(R')=r(R)+1$. We have the LLT-equivalence \begin{equation}q(R,R')+(R\cup R')\cong q(R\cup R')+(R',R).\end{equation}
\end{lemma}
\begin{proof} By Theorem \ref{thm:llteq}, it suffices to find an appropriate bijection \begin{equation}f:\text{SSYT}_{(R,R')}\sqcup\text{SSYT}_{(R\cup R')}\to\text{SSYT}_{(R\cup R')}\sqcup\text{SSYT}_{(R',R)}.\end{equation} For a tableau $\bm T\in\text{SSYT}_{(R,R')}$, let $x$ and $y$ denote the entries in the cells of content $r(R)$ and $\ell(R')$ respectively. We partition \begin{equation}\text{SSYT}_{(R,R')}=\text{SSYT}^\leq_{(R,R')}\sqcup\text{SSYT}^>_{(R,R')}\end{equation} depending on whether $x\leq y$ or $x>y$, and we similarly partition $\text{SSYT}_{(R',R)}$. We will now assemble our bijection $f$ as the union of the unique bijections \begin{align}f_1&:\text{SSYT}^>_{(R,R')}\to\text{SSYT}^>_{(R',R)},\\f_2&:\text{SSYT}^\leq_{(R,R')}\to\text{SSYT}_{(R\cup R')},\text{ and }\\f_3&:\text{SSYT}_{(R\cup R')}\to\text{SSYT}^\leq_{(R',R)}\end{align} satisfying \eqref{eq:contentpreserving}. By definition, every $\bm T\in\text{SSYT}^>_{(R',R)}$ has the unique inversion $(y,z)$, while all of the other tableaux have zero inversions, so our aggregate function $f$ satisfies \eqref{eq:inversionobeying}.
\end{proof}
\begin{example} These bijections are illustrated in an example below. We have written $q$'s to indicate how the numbers of inversions are supposed to change.
$$\begin{tabular}{ccccc}
$f_1:$ &\hspace{20pt} $q \ \tableau{&&y&z\\ \\ w&x}$ &\hspace{20pt} $\mapsto$ &\hspace{20pt} $\tableau{w&x\\ \\ &&y&z}$ &\hspace{20pt} $(x>y)$ \\ &&&& \\ \hline &&&& \\ $f_2:$ &\hspace{20pt} $q \ \tableau{&&y&z\\ \\ w&x}$ &\hspace{20pt} $\mapsto$ &\hspace{20pt} $q \ \tableau{w&x&y&z}$ &\hspace{20pt} $(x\leq y)$\\&&&& \\ \hline &&&&\\ $f_3:$ &\hspace{20pt} $\tableau{w&x&y&z}$ &\hspace{20pt} $\mapsto$ &\hspace{20pt} $\tableau{&&y&z\\ \\ w&x}$ &\hspace{20pt} $(x\leq y)$
\end{tabular}$$
\end{example}

Our next LLT-equivalence relation justifies the terminology of commuting rows. 
\begin{lemma} \label{lem:commuting} Let $R$ and $R'$ be rows such that $R\leftrightarrow R'$. We have the LLT-equivalence \begin{equation}(R,R')\cong(R',R).\end{equation} \end{lemma} 
\begin{remark} Lemma \ref{lem:commuting} tells us that if $\bm\lambda=(R_1,\ldots,R_n)$ and $R_i\leftrightarrow R_{i+1}$, then we can switch rows $R_i$ and $R_{i+1}$ so that $G_{\bm\lambda}(\bm x;q)=G_{\bm\mu}(\bm x;q)$, where $\bm\mu=(R_1,\ldots,R_{i+1},R_i,\ldots,R_n)$. Also note that Theorem \ref{thm:maxinv} implies that if $(R,R')\cong(R',R)$, then $R\leftrightarrow R'$. \end{remark}

\begin{proof}
By Theorem \ref{thm:llteq}, it suffices to find an appropriate bijection \begin{equation}f:\text{SSYT}_{(R,R')}\to\text{SSYT}_{(R',R)}.\end{equation} By Proposition \ref{prop:mrirj}, we have $M(R,R')=0$ or $\min\{|R|,|R'|\}$. If $M(R,R')=0$, then the unique bijection $f$ satisfying \eqref{eq:contentpreserving} trivially satisfies \eqref{eq:inversionobeying}, so suppose without loss of generality that $M(R,R')=|R'|$. Denote by $u_c$ and $v_c$ the cells in $R$ and $R'$ of content $c$. Let $\bm T\in\text{SSYT}_{(R,R')}$. First suppose that either $\bm T(v_i)<\bm T(u_{i-1})$ for all $i\in c(R')$ or $\bm T(v_j)>\bm T(u_{j+1})$ for all $j\in c(R')$, where by convention $\bm T(u_i)=0$ if $i<\ell(R)$ and $\bm T(u_i)=\infty$ if $i>r(R)$. Define $\bm U=f(\bm T)\in\text{SSYT}_{(R',R)}$ by $\bm U(u_c)=\bm T(u_c)$ and $\bm U(v_c)=\bm T(v_c)$. Then $f$ satisfies \eqref{eq:inversionobeying} because $\text{inv}(\bm U)=\text{inv}(\bm T)=|R'|$ and $f$ satisfies \eqref{eq:contentpreserving}, so we are done.\\

Now suppose otherwise and let $i\in c(R')$ be minimal with $\bm T(v_i)\geq \bm T(u_{i-1})$  and $j\in c(R')$ be maximal with $\bm T(v_j)\leq\bm T(u_{j+1})$. 
Define $\bm U=f(\bm T)\in\text{SSYT}_{(R',R)}$ by $\bm U(u_c)=\bm T(u_c)$ and $\bm U(v_c)=\bm T(v_c)$ if $c< i$ or $c> j$, and $\bm U(u_c)=\bm T(v_c)$ and $\bm U(v_c)=\bm T(u_c)$ if $i\leq c$ and $c\leq j$. An example is given below, where we have marked the entries of content $i$ and $j$ in red. Informally, the entries between the red ones are fixed and the entries outside are switched.
$$\bm T=\tableau{&&&1&1&\textcolor{red}5&5&\textcolor{red}6&8&9\\ \\1&2&3&3&4&\textcolor{red}4&6&\textcolor{red}7&7&7&8&8&9}\mapsto\bm U=\tableau{1&2&3&3&4&\textcolor{red}5&5&\textcolor{red}6&7&7&8&8&9\\ \\&&&1&1&\textcolor{red}4&6&\textcolor{red}7&8&9}$$
By construction, $\bm U\in\text{SSYT}_{(R',R)}$ and $f$ satisfies \eqref{eq:contentpreserving}. By minimality of $i$, $\bm T$ has inversions $(\bm T(u_c),\bm T(v_c))$ for $\ell(R')\leq c< i$ and by maximality of $j$, $\bm T$ has inversions $(\bm T(v_c),\bm T(u_{c+1}))$ for $j<c\leq r(R)$. Similarly, $\bm U$ has inversions $(\bm U(u_{c-1}), \bm U(v_c))$ for $\ell(R')\leq c< i$ and $(\bm U(v_c),\bm U(u_c))$ for $j< c\leq r(R')$. The four pairs $$(\bm T(v_{i-1}),\bm T(u_i)), \ (\bm T(v_j),\bm T(u_{j+1})), \ (\bm U(u_{i-1}),\bm U(v_i)),\text{ and }(\bm U(u_j),\bm U(v_{j+1}))$$ are not inversions, and otherwise if $i\leq c$ and $c\leq j$ the cells of content $c$ are unchanged so any inversions are preserved. Therefore, $f$ also satisfies \eqref{eq:inversionobeying}. We can recover the contents $i$ and $j$ from $\bm U$ by noting that $i\in c(R')$ is minimal with $\bm U(v_i)\geq \bm U(u_{i-1})$ and $j\in c(R')$ is maximal with $\bm U(v_j)\leq \bm U(u_{j+1})$, so $f$ is invertible. \end{proof}

By combining Lemma \ref{lem:llteq} and Lemma \ref{lem:commuting}, we obtain the following recurrence relation.

\begin{lemma} \label{lem:inductiverelation} Let $R$ and $R'$ be rows such that $R\nleftrightarrow R'$ and $\ell(R')<\ell(R)$. We have the LLT-equivalence \begin{equation}\label{eq:inductiverelation}(R,R')\cong q(R',R)+(1-q)(R\cup R',R\cap R').\end{equation}
\end{lemma}
\begin{proof} Note that by Proposition \ref{prop:mrirj}, we must have that $\ell(R')<\ell(R)\leq r(R')+1\leq r(R)$. Let $R_1=R'\setminus R$ and $R_2=R\cap R'$, so that $R'=R_1\cup R_2$, $R\cup R_1=R\cup R'$, and $R\leftrightarrow R_2$. Now by Lemma \ref{lem:llteq} and Lemma \ref{lem:commuting}, we have
\begin{align}(R,R')&\cong\frac q{q-1}(R,R_1,R_2)-\frac 1{q-1}(R,R_2,R_1)\\\nonumber&\cong\frac q{q-1}(q(R_1,R,R_2)+(1-q)(R_1\cup R,R_2))-\frac 1{q-1}(R_2,R,R_1)\\\nonumber&\cong\frac{q^2}{q-1}(R_1,R_2,R)-q(R_1\cup R,R_2)-\frac 1{q-1}(q(R_2,R_1,R)+(1-q)(R_2,R_1\cup R))\\\nonumber&\cong\frac{q^2}{q-1}(R_1,R_2,R)-q(R_1\cup R,R_2)-\frac q{q-1}(R_2,R_1,R)+(R_1\cup R,R_2)\\\nonumber&\cong\frac{q^2}{q-1}(R_1,R_2,R)+(1-q)(R_1\cup R,R_2)-\frac q{q-1}(q(R_1,R_2,R)+(1-q)(R',R))\\\nonumber&= q(R',R)+(1-q)(R\cup R',R\cap R').
\end{align}
\end{proof}

The following Proposition will help us visualize the relation \eqref{eq:inductiverelation}. 

\begin{proposition} \label{prop:munin} Let $R_1$, $R_2$, and $R$ be rows such that $R_1\nleftrightarrow R_2$ and $\ell(R_2)<\ell(R_1)$ and let $M=M(R_1,R_2)$, $M_1=M(R_1,R)$, and $M_2=M(R_2,R)$. Then \begin{align}M(R_1\cap R_2,R)&=\min\{M-1,M_1,M_2\}\text{ and }\\M(R_1\cup R_2,R)&=\min\{|R|,\max\{M_1,M_2,M_1+M_2-(M-1)\}\}.\end{align} In particular, if $M_2=0$, then $M(R_1\cap R_2,R)=0$ and $M(R_1\cup R_2,R)=M_1$. 
\end{proposition}

\begin{proof} Recall that by Proposition \ref{prop:mrirj}, we must have $\ell(R_1)-1\leq r(R_2)$. We now consider several cases. If $\ell(R)\geq\ell(R_1)$, then \begin{equation}M(R_1\cap R_2,R)=M(R_2,R)=M_2\leq M-1\leq M(R_1\cup R_2,R)=M(R_1,R)=M_1.\end{equation} Similarly, if $r(R)<r(R_2)-1$, then \begin{equation}M(R_1\cap R_2,R)=M(R_1,R)=M_1\leq M-1\leq M(R_1\cup R_2,R)=M(R_2,R)=M_2.\end{equation} Now suppose that $\ell(R)\leq\ell(R_1)-1\leq r(R_2)\leq r(R)$, so that $R_1\cap R_2\subseteq R$ and $M(R_1\cap R_2,R)=|R_1\cap R_2|=M-1\leq M_1,M_2$. If $\ell(R)<\ell(R_2)$, then \begin{equation}M(R_1\cup R_2,R)=|R_1\cap R^+|+|R_2\cap R^+|-|R_1\cap R_2\cap R^+|=M_1+M_2-(M-1)\leq |R|.\end{equation} Similarly, if $\ell(R)\geq\ell(R_2)$, then \begin{equation}M(R_1\cup R_2,R)=|R_1\cap R|+|R_2\cap R|-|R_1\cap R_2\cap R|=|R_1\cap R|+M_2-(M-1).\end{equation} At this point, if $r(R)\geq r(R_1)$, then $|R_1\cap R|=M_1$ and again \begin{equation}M(R_1\cup R_2,R)=M_1+M_2-(M-1)\leq |R|,\end{equation} while if $r(R)<r(R_1)$, then $|R_1\cap R|=M_1-1$ and $R\subseteq R_1\cup R_2$, so \begin{equation}M(R_1\cup R_2,R)=|R|=(M_1-1)+M_2-(M-1)<M_1+M_2-(M-1).\end{equation}
\end{proof}

\begin{remark} \label{rem:changegraph} We can now visualize the relation \eqref{eq:inductiverelation} as below, where $c=a+b-(M-1)$, $N_1=\min\{M-1,M_1,M_2\}$, and $N_2=\min\{x,\max\{M_1,M_2,M_1+M_2-(M-1)\}\}$.\\ 

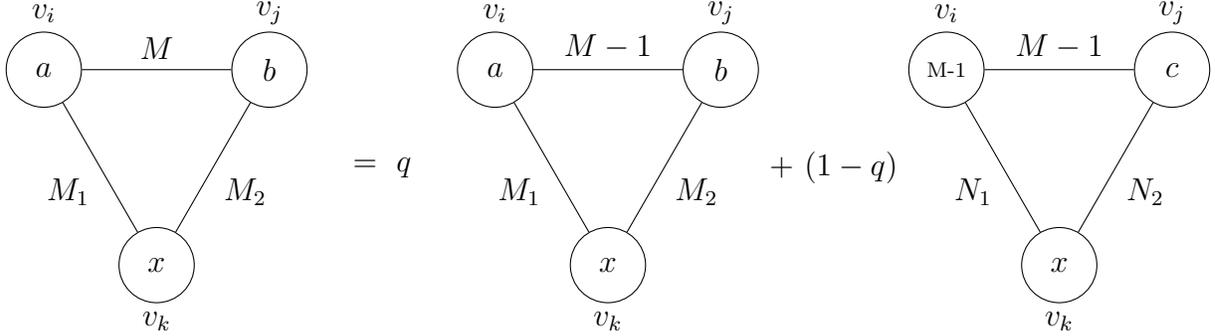
\begin{tikzpicture}
\node[shape=circle,draw=black,minimum size=10mm](1) at (0,0) {$a$};
\node[shape=circle,draw=black,minimum size=10mm](2) at (3,0) {$b$};
\node[shape=circle,draw=black,minimum size=10mm](3) at (1.5,-2.6) {$x$};
\draw (0,0.75) node {$v_i$} (3,0.75) node {$v_j$} (1.5,-3.35) node {$v_k$};
\path [-](1) edge node [above]{$M$} (2);
\path [-](1) edge node [below left]{$M_1$} (3);
\path [-](2) edge node [below right]{$M_2$} (3);
\draw (4.5,-1.3) node {$= \ q$};
\node[shape=circle,draw=black,minimum size=10mm](1) at (6,0) {$a$};
\node[shape=circle,draw=black,minimum size=10mm](2) at (9,0) {$b$};
\node[shape=circle,draw=black,minimum size=10mm](3) at (7.5,-2.6) {$x$};
\draw (6,0.75) node {$v_i$} (9,0.75) node {$v_j$} (7.5,-3.35) node {$v_k$};
\path [-](1) edge node [above]{$M-1$} (2);
\path [-](1) edge node [below left]{$M_1$} (3);
\path [-](2) edge node [below right]{$M_2$} (3);
\draw (10.5,-1.3) node {$+ \ (1-q)$};
\node[shape=circle,draw=black,minimum size=10mm](1) at (12,0) {\tiny{M-1}};
\node[shape=circle,draw=black,minimum size=10mm](2) at (15,0) {$c$};
\node[shape=circle,draw=black,minimum size=10mm](3) at (13.5,-2.6) {$x$};
\draw (12,0.75) node {$v_i$} (15,0.75) node {$v_j$} (13.5,-3.35) node {$v_k$};
\path [-](1) edge node [above]{$M-1$} (2);
\path [-](1) edge node [below left]{$N_1$} (3);
\path [-](2) edge node [below right]{$N_2$} (3);
\end{tikzpicture}

\end{remark}

We conclude this section by characterizing the triangle-free weighted graphs $\Pi(\bm\lambda)$ that can arise from a horizontal-strip $\bm\lambda$.

\begin{definition} A graph $G$ is a \emph{caterpillar} if it is a tree and its vertices can be partitioned $V=P\sqcup L$ so that the induced subgraph $G[P]$ is a path and every $v\in L$ has degree one. \end{definition}

\begin{proposition} \label{prop:caterpillar} Let $\bm\lambda$ be a horizontal-strip such that $\Pi=\Pi(\bm\lambda)$ is triangle-free. Recall that for a vertex $v\in \Pi$, we denote by $|v|$ the size of the corresponding row of $\bm\lambda$. 
\begin{enumerate} 
\item If $i<j<k$ and $v_i$ is adjacent to $v_k$, then $M_{j,k}=|v_j|$. 
\item Every vertex $v_i$ is adjacent to at most one vertex $v_j$ for which $i<j$.
\item Every connected component of $\Pi$ is a caterpillar $C=(P\sqcup L,E)$ and if $v_j\in L$ is adjacent to $v_k$, then $M_{j,k}=|v_j|$. 
\item If $v_i$ is adjacent to the vertices $\{v_{j_t}\}_{t=1}^r$, then \begin{equation}|v_i|+1\geq\sum_{t=1}^r M_{i,j_t}.\end{equation}
\item Let $\Pi'$ be a graph whose vertices $\{v_1,\ldots,v_n\}$ have positive integer weights $|v_i|$ and whose edges $(v_i,v_j)$ have positive integer weights $M_{i,j}\leq\min\{|v_i|,|v_j|\}$. If $\Pi'$ satisfies the above four conditions, then $\Pi'=\Pi(\bm\mu)$ for some horizontal-strip $\bm\mu$. 
\end{enumerate}
\end{proposition}

\begin{remark} In particular, Part 5 tells us that the property in Part 1 precisely characterizes the labellings of $\Pi$ that can arise from a horizontal-strip. \end{remark}

\begin{example} \label{ex:caterpillar} For the horizontal-strip $\bm\lambda$ below left, we have drawn the caterpillar $\Pi(\bm\lambda)$ below right. We have $P=\{v_1,v_4,v_6\}$ and note that $8+1\geq 3+2+2+2$ and $4+1\geq 2+3$.
$$
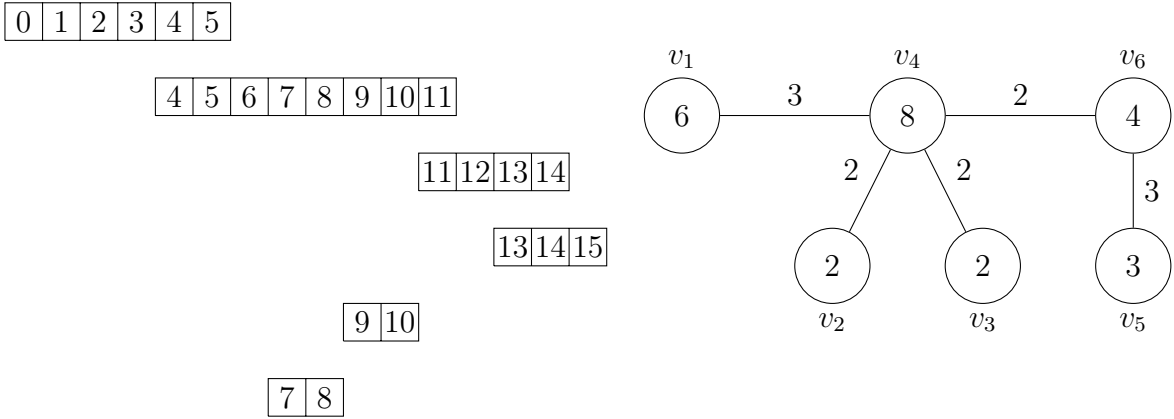
\begin{tikzpicture}
\node[shape=circle,draw=black,minimum size=10mm] (3) at (9,6) {6};
\draw (9,6.75) node {$v_1$};
\node[shape=circle,draw=black,minimum size=10mm] (6) at (12,6) {8};
\draw (12,6.75) node {$v_4$};
\node[shape=circle,draw=black,minimum size=10mm] (8) at (15,6) {4};
\draw (15,6.75) node {$v_6$};
\node[shape=circle,draw=black,minimum size=10mm] (4) at (11,4) {2};
\draw (11,3.25) node {$v_2$};
\node[shape=circle,draw=black,minimum size=10mm] (5) at (13,4) {2};
\draw (13,3.25) node {$v_3$};
\node[shape=circle,draw=black,minimum size=10mm] (7) at (15,4) {3};
\draw (15,3.25) node {$v_5$};
\path [-](3) edge node [above]{$3$} (6);
\path [-](6) edge node [above]{$2$} (8);
\path [-](4) edge node [above left]{$2$} (6);
\path [-](5) edge node [above right]{$2$} (6);
\path [-](7) edge node [right]{$3$} (8);
\draw (3.75,2.25) node {7} (4.25,2.25) node {8} (4.75,3.25) node {9} (5.25,3.25) node {10} (6.75,4.25) node {13} (7.25,4.25) node {14} (7.75,4.25) node {15} (5.75,5.25) node {11}
 (6.25,5.25) node {12} (6.75,5.25) node {13} (7.25,5.25) node {14}
 (2.25,6.25) node {4} (2.75,6.25) node {5} (3.25,6.25) node {6}
 (3.75,6.25) node {7} (4.25,6.25) node {8} (4.75,6.25) node {9}
 (5.25,6.25) node {10} (5.75,6.25) node {11} (0.25,7.25) node {0} (0.75,7.25) node {1} (1.25,7.25) node {2} (1.75,7.25) node {3}
 (2.25,7.25) node {4} (2.75,7.25) node {5};
\draw (3.5,2) -- (4.5,2)--(4.5,2.5)--(3.5,2.5)--(3.5,2) (4,2)--(4,2.5) (4.5,3)--(5.5,3)--(5.5,3.5)--(4.5,3.5)--(4.5,3) (5,3)--(5,3.5) (6.5,4)--(8,4)--(8,4.5)--(6.5,4.5)--(6.5,4) (7,4)--(7,4.5) (7.5,4)--(7.5,4.5) (5.5,5)--(7.5,5)--(7.5,5.5)--(5.5,5.5)--(5.5,5) (6,5)--(6,5.5) (6.5,5)--(6.5,5.5) (7,5)--(7,5.5) (2,6)--(6,6)--(6,6.5)--(2,6.5)--(2,6) (2.5,6)--(2.5,6.5) (3,6)--(3,6.5) (3.5,6)--(3.5,6.5) (4,6)--(4,6.5) (4.5,6)--(4.5,6.5) (5,6)--(5,6.5) (5.5,6)--(5.5,6.5) (0,7)--(3,7)--(3,7.5)--(0,7.5)--(0,7) (0.5,7)--(0.5,7.5) (1,7)--(1,7.5) (1.5,7)--(1.5,7.5) (2,7)--(2,7.5) (2.5,7)--(2.5,7.5);
\end{tikzpicture}$$
By Part 5, we could switch the labels of $v_5$ and $v_6$ and the resulting graph $\Pi'$ would still be of the form $\Pi(\bm\mu)$ for some horizontal-strip $\bm\mu$. In this case, we would have $P=\{v_1,v_4,v_5,v_6\}$. However, because of Part 1, we could not switch the labels of $v_1$ and $v_2$.
\end{example}

\begin{proof}[Proof of Proposition \ref{prop:caterpillar}]
Let $\bm\lambda=(R_1,\ldots,R_n)$. 
\begin{enumerate}
\item Let $R_{i'}$, $R_{j'}$, and $R_{k'}$ be the rows of $\bm\lambda$ corresponding to the vertices $v_i$, $v_j$, and $v_k$ respectively. By Proposition \ref{prop:lltfacts}, Part 3, we may assume without loss of generality that $i'=1$. Because $i<j<k$, we have $r(R_{i'})\leq r(R_{j'})\leq r(R_{k'})$, and if $v_i$ is adjacent to $v_k$, then $\ell(R_{k'})\leq r(R_{i'})$, so $r(R_{j'})\in c(R_{k'})$ and $v_j$ is adjacent to $v_k$. Now using that $\Pi$ is triangle-free, $v_i$ is not adjacent to $v_j$, so $r(R_{i'})<\ell(R_{j'})$, $R_{j'}\subseteq R_{k'}$, and $M_{j,k}=|v_j|$. \\

\item If $v_i$ is adjacent to $v_j$ and $v_k$ with $i<j$ and $i<k$, then $v_j$ and $v_k$ are adjacent by Part 1, creating a triangle in $\Pi$. \\

\item We first note that $\Pi$ must be acyclic because if the vertices $\{v_{i_t}\}_{t=1}^r$ with $i_1<\cdots<i_r$ form a cycle, then the vertex $v_{i_1}$ is adjacent to two vertices $v_{i_t}$ and $v_{i_{t'}}$ with $i_1<i_t,i_{t'}$, contradicting Part 2. Therefore, the connected components of $\Pi$ must be trees. Let $C=(V,E)$ be a connected component of $\Pi$ and note that by Part 1, we must have $V=\{v_i: \ i_1\leq i\leq i_r\}$ for some $i_1,i_r$. Because $C$ is connected, there must be a path $P=(v_{i_1},v_{i_2},\ldots,v_{i_{r-1}},v_{i_r})$ and it follows from Part 2 that $i_1<i_2<\cdots<i_{r-1}<i_r$. By Part 1, if $i_t<j<i_{t+1}$, then $v_j$ must be adjacent to $v_{i_{t+1}}$. Because $C$ is a tree, this accounts for all of the edges of $C$ so indeed $C$ is a caterpillar. Moreover, by Part 1, if $v_i\in L=C\setminus P$ is adjacent to $v_j$, then $M_{i,j}=|v_i|$.\\

\item Let $R_{i'}$ and $R_{j'_t}$ be the rows of $\bm\lambda$ corresponding to vertices $v_i$ and $v_{j_t}$ respectively, and again by Proposition \ref{prop:lltfacts}, Part 3, we may assume that $i'=1$. Because $\Pi$ is triangle-free, we have $M_{j_t,j_{t'}}=0$ for $t\neq t'$ so assuming without loss of gnerality that $j_1<\cdots<j_r$, we have $\ell(R_{j'_{t+1}})\geq r(R_{j'_t})+1$ for every $t$. If $\ell(R_{j'_t}),\ell(R_{j'_{t'}})<\ell(R_1)$ or if $r(R_{j'_t}),r(R_{j'_{t'}})\geq r(R_1)$, then $v_t$ and $v_{t'}$ are adjacent, so we must have $\ell(R_{j'_t})\geq\ell(R_1)$ for every $t\geq 2$ and $r(R_{j'_t})\leq r(R_1)-1$ for every $t\leq r-1$. Therefore
\begin{align}M_{i,j_1}&\leq r(R_{j'_1})-\ell(R_1)+2,\\M_{i,j_2}&=|R_{j'_2}|=r(R_{j'_2})-\ell(R_{j'_1})+1\leq r(R_{j'_2})-r(R_{j'_1}),\\\nonumber&\cdots\\M_{i,j_{r-1}}&=|R_{j'_{r-1}}|=r(R_{j'_{r-1}})-\ell(R_{j'_{r-2}})+1\leq r(R_{j'_{r-1}})-r(R_{j'_{r-2}}),\\M_{i,j_r}&\leq r(R_1)-\ell(R_{j'_r})+1\leq r(R_1)-r(R_{j'_{r-1}}),
\end{align} and summing we have \begin{equation}\sum_{t=1}^rM_{i,j_t}\leq r(R_1)-\ell(R_1)+2= |R_1|+1=|v_i|+1.\end{equation}

\item It suffices to construct a horizontal-strip $\bm\mu$ for each connected component $C=(P\sqcup L,E)$ because in general we can translate each $\bm\mu$ to avoid attacking the others. Let $P=\{v_{i_1},\ldots,v_{i_r}\}$ with $i_1<\cdots<i_r$ and define the integers \begin{equation}a_t=\sum_{t'=1}^{t-1}(|v_{i_{t'}}|-M_{i_{t'},i_{t'+1}}+1), \ b_t=a_t+|v_{i_t}|-1,\text{ and }c_{t,t'}=\sum_{i_t+1\leq k\leq i_t+t'}|v_k|.\end{equation}
Now our desired horizontal-strip is \begin{equation}\bm\mu=(S_{i_1+1},\ldots,S_{i_2-1},S_{i_2+1},\ldots,S_{i_3-1},S_{i_3+1},\ldots,S_{i_r-1},S_{i_r},\ldots,S_{i_1}),\end{equation} where $\ell(S_j)=a_t$ and $r(S_j)=b_t$ if $j=i_t$ and $\ell(S_j)=b_{t-1}+c_{t,t'-1}+2$ and $r(S_j)=b_{t-1}+c_{t,t'}+1$ if $i_{t-1}<j=i_{t-1}+t'<i_t$. Indeed, we have that $r(S_j)$ is increasing, $|S_j|=|v_j|$, $M(S_{i_{t+1}},S_{i_t})=M_{i_t,i_{t+1}}$, $M(S_j,S_{i_t})=|v_j|$ for $i_{t-1}<j<i_t$, and we have $M(S_j,S_{j'})=0$ otherwise because by Part 4, we have for $i_{t-1}<j<i_t$ that \begin{align}r(S_{i_{t-1}})+1&\leq b_{t-1}+1<\ell(S_j)\leq r(S_j)\leq b_{t-1}+\sum_{i_t+1\leq k\leq i_{t+1}-1}|v'_k|+1\\\nonumber&\leq b_{t-1}+|S_{i_t}|-M_{i_{t-1},i_t}-M_{i_t,i_{t+1}}+2=a_{t+1}-1<\ell(S_{i_{t+1}}).\end{align} As an illustration, the horizontal-strip in Example \ref{ex:caterpillar} is the one constructed here.
\end{enumerate}\end{proof}

\section{The combinatorial formula}\label{sec:caterpillar}
In this section, we generalize cocharge in order to prove a combinatorial formula for the LLT polynomial $G_{\bm\lambda}(\bm x;q)$ whenever the weighted graph $\Pi(\bm\lambda)$ is triangle-free.

\begin{definition} Let $\mu$ be a partition and let $T\in\text{SSYT}_\mu$ be a tableau of shape $\mu$ with smallest entry $i$. We define the integer \begin{equation}f(T)=\max\{t: \ 0\leq t\leq \mu_1-\mu_2, \ t\leq w_i(T), \ T_{2,j'}>T_{1,j'+t} \text{ for all }1\leq j'\leq \mu_2\},\end{equation} where as before, $w_i(T)$ is the number of $i$'s in $T$. \end{definition}

\begin{remark} Informally, $f(T)$ is the maximum number of $i$'s that we can remove from $T$ so that no entry moves down when we rectify the resulting skew tableau. Alternatively, $f(T)$ is the maximum number of cells that we can shift the bottom row of $T$ to the left and still have the columns strictly increasing from bottom to top, as long as $f(T)\leq w_i(T)$. \end{remark}

\begin{definition} Let $T$ be an SSYT and let $i<j$ be integers. We denote by $T\vert_{i,j}$ the rectification of the skew tableau obtained by restricting $T$ to the entries $x$ with $i\leq x\leq j$, and we define the integer \begin{equation}\text{cocharge}_{i,j}(T)=w_i(T)-f(T\vert_{i,j}).\end{equation}
\end{definition}

\begin{example} For the tableaux $S$ and $T$ below left, we have drawn the tableaux $S\vert_{2,4}$ and $T\vert_{2,4}$ below right. We have $f(S\vert_{2,4})=3$ and $\text{cocharge}_{2,4}(S)=5-3=2$, and we have $f(T\vert_{2,4})=3$ because we must have $t\leq w_2(T\vert_{2,4})$, so $\text{cocharge}_{2,4}(T)=3-3=0$. \begin{align*} &S=\tableau{3&4\\2&2&3&4&5&5\\1&1&1&2&2&2&3&4} \ S\vert_{2,4}=\tableau{4\\3&3&4\\2&2&2&2&2&3&4}\\ \\&T=\tableau{4&5\\2&2&4&4&5&5\\1&1&1&2&3&3&3&3} \ T\vert_{2,4}=\tableau{\\4&4&4\\2&2&2&3&3&3&3}\\\end{align*}\end{example}

\begin{example} In the case where $j=i+1$, $\text{cocharge}_{i,i+1}(T)$ is the number of entries in the second row of $T\vert_{i,i+1}$, which in Example \ref{ex:cochij} we saw is $\text{cocharge}(T\vert_{i,i+1})$. \end{example}

We now state our main Theorem.

\begin{theorem} \label{thm:caterpillar} Let $\bm\lambda$ be a horizontal-strip such that the weighted graph $\Pi=\Pi(\bm\lambda)$ is triangle-free and let $\alpha_i=|v_i|$. Then the LLT polynomial of $\bm\lambda$ is \begin{equation}\label{eq:caterpillar}G_{\bm\lambda}(\bm x;q)=\sum_{T\in\text{SSYT}(\alpha)}q^{\text{cocharge}_\Pi(T)}s_{\text{shape}(T)},\end{equation} where $\text{cocharge}_\Pi(T)=\sum_{i<j}\min\{M_{i,j},\text{cocharge}_{i,j}(T)\}$.\end{theorem}

Before we prove Theorem \ref{thm:caterpillar}, we will present some examples and special cases to familiarize ourselves with this generalization of cocharge.

\begin{example} \label{ex:theorem}
Let $\bm\lambda=(6/5,9/6,7/2,4/0)$ be the horizontal-strip below left with the rightmost cells in each row labelled in content reading order, so that $\Pi(\bm\lambda)$ is the caterpillar below right.

$$
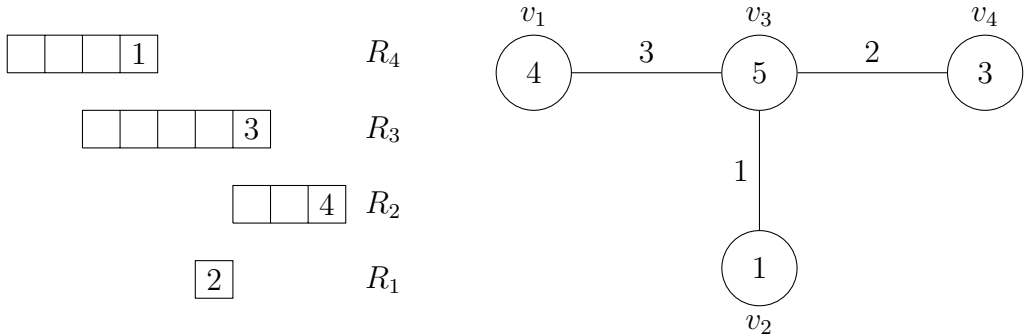
\begin{tikzpicture}
\node[shape=circle,draw=black,minimum size=10mm] (3) at (9,6) {4};
\draw (9,6.75) node {$v_1$};
\node[shape=circle,draw=black,minimum size=10mm] (6) at (12,6) {5};
\draw (12,6.75) node {$v_3$};
\node[shape=circle,draw=black,minimum size=10mm] (8) at (15,6) {3};
\draw (15,6.75) node {$v_4$};
\node[shape=circle,draw=black,minimum size=10mm] (4) at (12,3.4) {1};
\draw (12,2.65) node {$v_2$};
\path [-](3) edge node [above]{$3$} (6);
\path [-](6) edge node [above]{$2$} (8);
\path [-](4) edge node [left]{$1$} (6);
\draw (4.75,3.25) node {2} (6.25,4.25) node {4} (5.25,5.25) node {3}
 (3.75,6.25) node {1};
\draw (4.5,3)--(5,3)--(5,3.5)--(4.5,3.5)--(4.5,3)  (5,4)--(6.5,4)--(6.5,4.5)--(5,4.5)--(5,4) (5.5,4)--(5.5,4.5) (6,4)--(6,4.5) (3,5)--(5.5,5)--(5.5,5.5)--(3,5.5)--(3,5) (3.5,5)--(3.5,5.5) (4,5)--(4,5.5) (4.5,5)--(4.5,5.5) (5,5)--(5,5.5)  (2,6)--(4,6)--(4,6.5)--(2,6.5)--(2,6) (2.5,6)--(2.5,6.5) (3,6)--(3,6.5) (3.5,6)--(3.5,6.5);
\draw (7,3.25) node {$R_1$} (7,4.25) node {$R_2$} (7,5.25) node {$R_3$} (7,6.25) node {$R_4$};
\end{tikzpicture}$$

To calculate the coefficient of $s_{733}$, we consider the three tableaux of weight $\alpha=4153$ and shape $733$ as follows. The values of $\text{cocharge}_{\Pi}$ are calculated below. \begin{align*}T \ = \ &\tableau{4&4&4\\3&3&3\\1&1&1&1&2&3&3} \ T_1 \ = \ & &\tableau{4&4&4\\2&3&3\\1&1&1&1&3&3&3} \ T_2 \ = \ & &\tableau{3&4&4\\2&3&3\\1&1&1&1&3&3&4}\\ & 2+1+2=5 & & 3+0+2=5& & 3+1+2=6\end{align*}
Therefore, the coefficient of $s_{733}$ is $(q^6+2q^5)$. 
\end{example}

\begin{corollary} \label{cor:path} Let $\bm\lambda$ be a horizontal-strip whose weighted graph $\Pi(\bm\lambda)$ is the path below.
$$
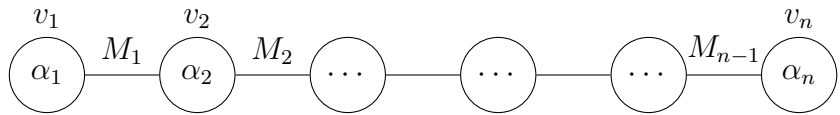
\begin{tikzpicture}
\node[shape=circle,draw=black,minimum size=10mm](1) at (0,0){$\alpha_1$};
\node[shape=circle,draw=black,minimum size=10mm](2) at (2,0){$\alpha_2$};
\node[shape=circle,draw=black,minimum size=10mm](3) at (4,0){$\cdots$};
\node[shape=circle,draw=black,minimum size=10mm](4) at (6,0){$\cdots$};
\node[shape=circle,draw=black,minimum size=10mm](5) at (8,0){$\cdots$};
\node[shape=circle,draw=black,minimum size=10mm](6) at (10,0){$\alpha_n$};
\draw (0,0.75) node {$v_1$} (2,0.75) node {$v_2$} (10,0.75) node {$v_n$};
\path [-](1) edge node [above]{$M_1$} (2);
\path [-](2) edge node [above]{$M_2$} (3);
\path [-](3) edge node {}(4);
\path [-](4) edge node {}(5);
\path [-](5) edge node [above]{$M_{n-1}$} (6);
\end{tikzpicture}$$
Then the LLT polynomial of $\bm\lambda$ is \begin{equation}G_{\bm\lambda}(\bm x;q)=\sum_{T\in\text{SSYT}(\alpha)}q^{\text{cocharge}_\Pi(T)}s_{\text{shape}(T)},\end{equation} where $\text{cocharge}_\Pi(T)=\sum_{i=1}^{n-1}\min\{M_i,\text{number of entries in the second row of }T\vert_{i,i+1}\}$.
\end{corollary}

\begin{remark} Corollary \ref{cor:path} generalizes \cite[Theorem 1]{lltmpath}, which gives a combinatorial formula for $G_{\bm\lambda}(\bm x;q)$ in certain cases where $\Pi(\bm\lambda)$ is a path as above, with the additional constraint that $M_{i-1}+M_i\leq\alpha_i$ for $2\leq i\leq n-1$. Note that by Proposition \ref{prop:caterpillar}, Part 4, the weaker inequality $M_{i-1}+M_i\leq\alpha_i+1$ holds for $2\leq i\leq n-1$. \end{remark}

\begin{example} Let $\bm\lambda$ be a horizontal-strip with exactly two rows, so that $\Pi(\bm\lambda)$ is $$
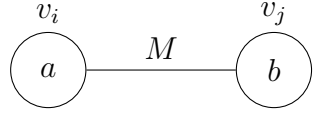
\begin{tikzpicture}
\node[shape=circle,draw=black,minimum size=10mm](1) at (0,0){$a$};
\node[shape=circle,draw=black,minimum size=10mm](2) at (3,0){$b$};
\draw (0,0.75) node {$v_i$} (3,0.75) node {$v_j$};
\path [-] (1) edge node [above]{$M$} (2);
\end{tikzpicture}$$
for some $a\geq b\geq M$, where $(i,j)$ is either $(1,2)$ or $(2,1)$, so $\alpha=(a,b)$ or $\alpha=(b,a)$ respectively. In either case, for each $0\leq k\leq b$, there is a unique tableau $T_k$ with content $\alpha$ and shape $(a+b-k)k$. Therefore, by Corollary \ref{cor:path}, the LLT polynomial is 
\begin{equation}G_{\bm\lambda}(\bm x;q)=\sum_{k=0}^bq^{\min\{M,k\}}s_{(a+b-k)k}=s_{(a+b)}+\cdots+q^Ms_{(a+b-M)M}+\cdots+q^Ms_{ab}.\end{equation}
Note that in this example, the formula \eqref{eq:caterpillar} does not depend on the labelling of $\Pi$. 
\end{example}

We now illustrate the idea of the proof in the case of Example \ref{ex:theorem}. 

\begin{proof}[Proof of Theorem \ref{thm:caterpillar} in the case of Example \ref{ex:theorem}. ]
We will use induction on $M(\bm\lambda)$. By applying Lemma \ref{lem:inductiverelation} to rows $R_3$ and $R_4$ and by using Proposition \ref{prop:munin}, we can write \begin{equation}G_{\bm\lambda}(\bm x;q)=qG_{\bm\lambda'}(\bm x;q)+(1-q)G_{\bm\lambda''}(\bm x;q),\end{equation} where the weighted graphs $\Pi=\Pi(\bm\lambda)$, $\Pi'=\Pi(\bm\lambda')$, and $\Pi''=\Pi(\bm\lambda'')$ are given below. $$
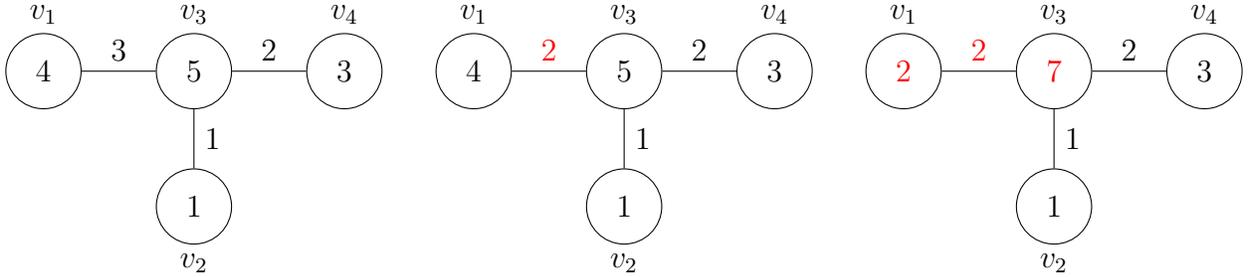
\begin{tikzpicture}
\node[shape=circle,draw=black,minimum size=10mm](1) at (0,0) {4};
\node at (0,0.75) {$v_1$};
\node[shape=circle,draw=black,minimum size=10mm](2) at (2,0) {5};
\node at (2,0.75) {$v_3$};
\node[shape=circle,draw=black,minimum size=10mm](3) at (2,-1.8) {1};
\node at (2,-2.55) {$v_2$};
\node[shape=circle,draw=black,minimum size=10mm](4) at (4,0) {3};
\node at (4,0.75) {$v_4$};
\path [-](1) edge node [above]{3} (2);
\path [-](2) edge node [right]{1} (3);
\path [-](2) edge node [above]{2} (4);
\end{tikzpicture} \hspace{20pt}
\begin{tikzpicture}
\node[shape=circle,draw=black,minimum size=10mm](1) at (0,0) {4};
\node at (0,0.75) {$v_1$};
\node[shape=circle,draw=black,minimum size=10mm](2) at (2,0) {5};
\node at (2,0.75) {$v_3$};
\node[shape=circle,draw=black,minimum size=10mm](3) at (2,-1.8) {1};
\node at (2,-2.55) {$v_2$};
\node[shape=circle,draw=black,minimum size=10mm](4) at (4,0) {3};
\node at (4,0.75) {$v_4$};
\path [-](1) edge node [above]{\textcolor{red}2} (2);
\path [-](2) edge node [right]{1} (3);
\path [-](2) edge node [above]{2} (4);
\end{tikzpicture} \hspace{20pt}
\begin{tikzpicture}
\node[shape=circle,draw=black,minimum size=10mm](1) at (0,0) {\textcolor{red}2};
\node at (0,0.75) {$v_1$};
\node[shape=circle,draw=black,minimum size=10mm](2) at (2,0) {\textcolor{red}7};
\node at (2,0.75) {$v_3$};
\node[shape=circle,draw=black,minimum size=10mm](3) at (2,-1.8) {1};
\node at (2,-2.55) {$v_2$};
\node[shape=circle,draw=black,minimum size=10mm](4) at (4,0) {3};
\node at (4,0.75) {$v_4$};
\path [-](1) edge node [above]{\textcolor{red}2} (2);
\path [-](2) edge node [right]{1} (3);
\path [-](2) edge node [above]{2} (4);
\end{tikzpicture}$$

By our induction hypothesis, our task is now to prove that \begin{align}\sum_{T\in\text{SSYT}(4153)}q^{\text{cocharge}_\Pi(T)}s_{\text{shape}(T)}=q&\sum_{T\in\text{SSYT}(4153)}q^{\text{cocharge}_{\Pi'}(T)}s_{\text{shape}(T)}\\\nonumber+ \ (1-q)&\sum_{S\in\text{SSYT}(2173)}q^{\text{cocharge}_{\Pi''}(S)}s_{\text{shape}(S)}.\end{align}

Let us consider the coefficient of $s_{733}$. The first sum on the right hand side corresponds to the three tableaux from Example \ref{ex:theorem} with the values of $\text{cocharge}_{\Pi'}$ calculated below.\begin{align*}T \ = \ &\tableau{4&4&4\\3&3&3\\1&1&1&1&2&3&3} \ T_1 \ = \ & &\tableau{4&4&4\\2&3&3\\1&1&1&1&3&3&3} \ T_2 \ = \ & &\tableau{3&4&4\\2&3&3\\1&1&1&1&3&3&4}\\ & 2+1+2=5 & & 2+0+2=4& & 2+1+2=5\end{align*} The factor of $q$ tells us to increase these values by one, and this corresponds to increasing $M_{1,3}$ from $2$ to $3$. Indeed, for the tableaux $T_1$ and $T_2$, because $\text{cocharge}_{1,3}(T_i)=3$, the contribution to $\text{cocharge}_\Pi(T_i)$ is now $\min\{3,3\}=3$ instead of $\min\{2,3\}=2$. However, the tableau $T$ has $\text{cocharge}_{1,3}(T)=2$, so in this case we do not want to increase the cocharge by one. The second sum allows us to make this correction. It corresponds to the tableau below. \begin{align*}S \ = \ &\tableau{4&4&4\\3&3&3\\1&1&2&3&3&3&3} \\ & 2+1+2=5 \end{align*} The factor of $(1-q)$ tells us to change the term $q^6s_{733}$ corresponding to $T$ back into the term $q^5s_{733}$. In general, this second sum precisely corrects for those tableaux for which we do not want to increase the cocharge by one when we increase $M_{1,3}$ by one. To be specific, we will define a bijection \begin{equation}\varphi:\text{SSYT}(2173)\to\{T\in\text{SSYT}(4153): \ \text{cocharge}_{1,3}(T)\leq 2\}\end{equation} such that $\text{cocharge}_{\Pi''}(S)=\text{cocharge}_{\Pi'}(\varphi(S))$. Informally, this bijection is given by changing two $3$'s on the bottom row of $S$ into $1$'s. We have defined $\text{cocharge}_{1,3}$ specifically to measure the property of being in the image of $\varphi$. In this case, this image is precisely those tableaux $T$ for which two $1$'s on the bottom row can be removed and two $3$'s can be added, meaning that $f(T)\geq 2$ and $\text{cocharge}_{1,3}(T)\leq 4-2=2$. Also note that by definition we have \begin{equation}\text{cocharge}_\Pi(T)=\begin{cases}\text{cocharge}_{\Pi'}(T),&\text{ if }\text{cocharge}_{1,3}(T)\leq 2,\\\text{cocharge}_{\Pi'}(T)+1,&\text{ if }\text{cocharge}_{1,3}(T)\geq 3.\end{cases}\end{equation}
Therefore, we have \begin{align} q&\sum_{T\in\text{SSYT}(4153)}q^{\text{cocharge}_{\Pi'}(T)}s_{\text{shape}(T)}+(1-q)\sum_{S\in\text{SSYT}(2173)}q^{\text{cocharge}_{\Pi''}(S)}s_{\text{shape}(S)}\\\nonumber=q&\sum_{\substack{T\in\text{SSYT}(4153)\\\text{cocharge}_{1,3}(T)\geq 3}}q^{\text{cocharge}_{\Pi'}(T)}s_{\text{shape}(T)}\\\nonumber+ \ q&\sum_{\substack{T\in\text{SSYT}(4153)\\\text{cocharge}_{1,3}(T)\leq 2}}q^{\text{cocharge}_{\Pi'}(T)}s_{\text{shape}(T)}+(1-q)\sum_{S\in\text{SSYT}(2173)}q^{\text{cocharge}_{\Pi''}(S)}s_{\text{shape}(S)}\\\nonumber=q&\sum_{\substack{T\in\text{SSYT}(4153)\\\text{cocharge}_{1,3}(T)\geq 3}}q^{\text{cocharge}_{\Pi'}(T)+1}s_{\text{shape}(T)}+\sum_{\substack{T\in\text{SSYT}(4153)\\\text{cocharge}_{1,j}(T)\leq 2}}q^{\text{cocharge}_{\Pi'}(T)}s_{\text{shape}(T)}\\\nonumber=&\sum_{T\in\text{SSYT}(4153)}q^{\text{cocharge}_\Pi(T)}s_{\text{shape}(T)},\end{align}as desired.\end{proof}

We will now make this argument precise and check the details. We first recall the Littlewood--Richardson rule, which gives a combinatorial formula for the product of two Schur functions.

\begin{theorem} \cite[Section 5.1, Corollary 2 and Corollary 3]{youngtab} (The Littlewood--Richardson rule) Let $\lambda$, $\mu$, and $\nu$ be partitions and fix a tableau $S\in\text{SSYT}_\mu$. Denote by $c^\lambda_{\mu,\nu}$ the number of tableaux $T\in\text{SSYT}_{\lambda/\nu}$ whose rectification is $S$. Then the product of Schur functions $s_\mu$ and $s_\nu$ expands in the basis of Schur functions as \begin{equation}s_\mu s_\nu=\sum_\lambda c^\lambda_{\mu,\nu}s_\lambda.\end{equation}\end{theorem}

We now prove Theorem \ref{thm:caterpillar}. \\

\begin{proof}[Proof of Theorem \ref{thm:caterpillar}. ]
Let $\bm\lambda=(R_1,\ldots,R_n)$ and let $R_{i'}$ be the row corresponding to the vertex $v_1$. We use induction on $n$. If $n=1$, then both sides of \eqref{eq:caterpillar} are $s_{\alpha_1}$, so we may assume that $n\geq 2$. \\

We first consider the case where the vertex $v_1$ has no neighbour. Let $\tilde \alpha=(0,\alpha_2,\ldots,\alpha_n)$ and note that we can associate a tableau $\tilde T\in\text{SSYT}_{\lambda/(\alpha_1)}(\tilde\alpha)$ with a tableau $T\in\text{SSYT}_\lambda(\alpha)$ by placing $\alpha_1$ $1$'s underneath $\tilde T$. Because $\text{cocharge}_\Pi$ is defined by restricting to the appropriate entries and rectifying, we have $\text{cocharge}_\Pi(T)=\text{cocharge}_\Pi(T\vert_{2,n})$. Now using Proposition \ref{prop:lltfacts}, Part 2, our induction hypothesis, and the Littlewood--Richardson rule, we have

\begin{align}
G_{\bm\lambda}(\bm x;q)&=G_{(R_{i'})}(\bm x;q)G_{(R_1,\ldots,R_{i'-1},R_{i'+1},\ldots,R_n)}(\bm x;q)=\sum_{S\in\text{SSYT}(\tilde\alpha)}q^{\text{cocharge}_\Pi(S)}s_{\text{shape}(S)}s_{\alpha_1}\\\nonumber&=\sum_{S\in\text{SSYT}(\tilde\alpha)}\sum_{\substack{T\in\text{SSYT}(\alpha) \\ T\vert_{2,n}=S}}q^{\text{cocharge}_\Pi(T)}s_{\text{shape}(T)}=\sum_{T\in\text{SSYT}(\alpha)}q^{\text{cocharge}_\Pi(T)}s_{\text{shape(T)}},
\end{align}

as desired. So we now suppose that the vertex $v_1$ has a neighbour $v_j$ corresponding to some row $R_{j'}$ and note that by Proposition \ref{prop:caterpillar}, Part 2, this neighbour is unique. We also use induction on $M(\bm\lambda)$. If $M(\bm\lambda)=0$, then the vertex $v_1$ has no neighbour, so we may assume that $M(\bm\lambda)\geq 1$. Using Proposition \ref{prop:caterpillar}, Part 2 again, the vertex $v_j$ has at most one neighbour $v_k$ for which $j<k$. By Proposition \ref{prop:lltfacts}, Part 2, we may assume that $i'>j'$. We claim that we may further assume that $i'=j'+1$ and that $R_{i'}\nleftrightarrow R_{j'}$. \\

For $1\leq t<i'$ with $t\neq j'$, because $M(R_t,R_{i'})=0$ and $r(R_{i'})\leq r(R_t)$, we must in fact have $r(R_{i'})<\ell(R_t)-1$ and $R_{i'}\leftrightarrow R_t$ by Proposition \ref{prop:mrirj}. Therefore, by Lemma \ref{lem:commuting}, we may move row $R_{i'}$ down to assume without loss of generality that $i'=j'+1$. If $R_{i'}\nleftrightarrow R_{j'}$, then we have established our claim. Otherwise, if $R_{i'}\leftrightarrow R_{j'}$, then we continue to use Lemma \ref{lem:commuting} to move the row $R_{i'}$ down, then we use Proposition \ref{prop:lltfacts}, Part 2, to move $R_{i'}$ back to the top and decrease $\ell(R_{i'})$ by one, and we use Lemma \ref{lem:commuting} to again assume that $i'=j'+1$. If $R_{i'}\nleftrightarrow R_{j'}$ then we are done, otherwise we continue this process, decreasing $\ell(R_{i'})$ by one each time. Because $R_{i'}$ and $R_{j'}$ will not commute when $\ell(R_{i'})<\ell(R_{j'})$, this process will eventually terminate and we may assume that $i'=j'+1$ and $R_{i'}\nleftrightarrow R_{j'}$. Now we can apply Lemma \ref{lem:inductiverelation} in order to write \begin{equation}G_{\bm\lambda}(\bm x;q)=qG_{\bm\lambda'}(\bm x;q)+(1-q)G_{\bm\lambda''}(\bm x;q),\end{equation} where $\bm\lambda'=(R_1,\ldots,R_{i'},R_{j'},\ldots,R_n)$ and $\bm\lambda''=(R_1,\ldots,R_{i'}\cup R_{j'},R_{i'}\cap R_{j'},\ldots,R_n)$. By Proposition \ref{prop:munin}, the graphs $\Pi=\Pi(\bm\lambda)$, $\Pi'=\Pi(\bm\lambda')$, and $\Pi''=\Pi(\bm\lambda'')$ are as below, where $M=M_{1,j}$ and $c=\alpha_1+\alpha_j-(M-1)$. 

$$
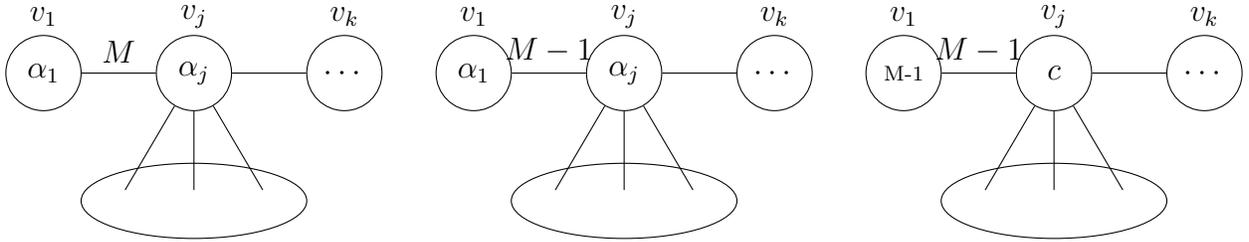
\begin{tikzpicture}
\node[shape=circle,draw=black,minimum size=10mm](1) at (0,0) {$\alpha_1$};
\node at (0,0.75) {$v_1$};
\node[shape=circle,draw=black,minimum size=10mm](2) at (2,0) {$\alpha_j$};
\node at (2,0.75) {$v_j$};
\node (3) at (1,-1.7){};
\node (4) at (2,-1.7){};
\node (5) at (3,-1.7){};
\node[ellipse, minimum height=10mm,minimum width=30mm,draw] at (2,-1.7) {};
\node[shape=circle,draw=black,minimum size=10mm](6) at (4,0) {$\cdots$};
\node at (4,0.75) {$v_k$};
\path [-](1) edge node [above]{$M$} (2);
\path [-](2) edge node []{} (3);
\path [-](2) edge node []{} (4);
\path [-](2) edge node []{} (5);
\path [-](2) edge node []{} (6);
\end{tikzpicture} \hspace{20pt} \begin{tikzpicture}
\node[shape=circle,draw=black,minimum size=10mm](1) at (0,0) {$\alpha_1$};
\node at (0,0.75) {$v_1$};
\node[shape=circle,draw=black,minimum size=10mm](2) at (2,0) {$\alpha_j$};
\node at (2,0.75) {$v_j$};
\node (3) at (1,-1.7){};
\node (4) at (2,-1.7){};
\node (5) at (3,-1.7){};
\node[ellipse, minimum height=10mm,minimum width=30mm,draw] at (2,-1.7) {};
\node[shape=circle,draw=black,minimum size=10mm](6) at (4,0) {$\cdots$};
\node at (4,0.75) {$v_k$};
\path [-](1) edge node [above]{$M-1$} (2);
\path [-](2) edge node []{} (3);
\path [-](2) edge node []{} (4);
\path [-](2) edge node []{} (5);
\path [-](2) edge node []{} (6);
\end{tikzpicture} \hspace{20pt} \begin{tikzpicture}
\node[shape=circle,draw=black,minimum size=10mm](1) at (0,0) {\tiny{M-1}};
\node at (0,0.75) {$v_1$};
\node[shape=circle,draw=black,minimum size=10mm](2) at (2,0) {$c$};
\node at (2,0.75) {$v_j$};
\node (3) at (1,-1.7){};
\node (4) at (2,-1.7){};
\node (5) at (3,-1.7){};
\node[ellipse, minimum height=10mm,minimum width=30mm,draw] at (2,-1.7) {};
\node[shape=circle,draw=black,minimum size=10mm](6) at (4,0) {$\cdots$};
\node at (4,0.75) {$v_k$};
\path [-](1) edge node [above]{$M-1$} (2);
\path [-](2) edge node []{} (3);
\path [-](2) edge node []{} (4);
\path [-](2) edge node []{} (5);
\path [-](2) edge node []{} (6);
\end{tikzpicture}$$

Let $\beta=(M-1,\alpha_2,\ldots,\alpha_{j-1},c,\alpha_{j+1},\ldots,\alpha_n)$ and let $t=\alpha_1-(M-1)$. We define a map \begin{equation}\varphi:\text{SSYT}(\beta)\to\{T\in\text{SSYT}(\alpha): \ \text{cocharge}_{1,j}(T)\leq M-1\}\end{equation}
as follows. For $S\in\text{SSYT}(\beta)$, let $\varphi(S)=T$ be the tableau of the same shape given by \begin{equation}T_{i',j'}=\begin{cases} 1 & \text{ if }i'=1\text{ and }j'\leq t,\\ S_{1,j'-t} & \text{ if }i'=1\text{ and }t<j'\leq j_1, \\S_{i',j'} & \text{ otherwise},\end{cases}\end{equation} where $j_1$ is the column of the rightmost $j$ in $S$. As an illustration, for $t=2$ and the tableaux $S_1$ and $S_2$ below left, we have $T_1=\varphi(S_1)$ and $T_2=\varphi(S_2)$ below right. Informally, we change the two red $3$'s on the bottom row into $1$'s. \begin{align*}S_1&=\tableau{3&4&5&5&5\\2&3&3&4&4&5\\1&1&1&2&3&3&3&3&\textcolor{red}3&\textcolor{red}3&5} \ &T_1&=\tableau{3&4&5&5&5\\2&3&3&4&4&5\\\textcolor{red}1&\textcolor{red}1&1&1&1&2&3&3&3&3&5} \\ \\ S_2&=\tableau{3&5&5\\2&3&3&4&4&5&5&5\\1&1&1&2&3&3&3&3&\textcolor{red}3&\textcolor{red}3&5} &T_2&=\tableau{3&5&5\\2&3&3&4&4&5&5&5\\\textcolor{red}1&\textcolor{red}1&1&1&1&2&3&3&3&3&5}\end{align*}
We first show that the map $\varphi$ is well-defined and is a bijection. Let $n_j$ denote the number of $j$'s in $S$ that are not on the bottom row. Because the columns of $S$ are strictly increasing, $n_j$ is at most the number of entries in $S$ less than $j$, that is $n_j\leq (M-1)+\alpha_2+\cdots+\alpha_{j-1}$. By Proposition \ref{prop:caterpillar}, Part 3 and Part 4, this means that $n_j\leq \alpha_j$, so $S$ has at least $c-\alpha_j=t$ $j$'s on the bottom row and indeed $t$ $j$'s have been replaced by $t$ $1$'s and $\varphi(T)\in\text{SSYT}(\alpha)$. Furthermore, by construction, $f(T)=f(S)+t\geq t$, so $\text{cocharge}_{1,j}(T)=\text{cocharge}_{1,j}(S)\leq \alpha_1-t=M-1$. Also, given $T\in\text{SSYT}(\beta)$ such that $\text{cocharge}_{1,j}(T)\leq M-1$, then $f(T)\geq t$ and we can define $S=\varphi^{-1}(T)\in\text{SSYT}(\alpha)$ by \begin{equation}S_{i',j'}=\begin{cases} T_{1,j'+t} & \text{ if }i'=1\text{ and }j'\leq j_1-t,\\ j & \text{ if }i'=1\text{ and }j_1-t<j'\leq j_1, \\T_{i',j'} & \text{ otherwise.}\end{cases}\end{equation} We now claim that $\text{cocharge}_{\Pi''}(S)=\text{cocharge}_{\Pi'}(T)$. We have already shown that $f(T)=f(S)+t$ and therefore $\text{cocharge}_{1,j}(S)=\text{cocharge}_{1,j}(T)$. For $2\leq i\leq j-1$, the tableau $S\vert_{i,j}$ is simply $T\vert_{i,j}$ with $t$ $j$'s appended to the right of the first row, and therefore $f(S\vert_{i,j})=f(T\vert_{i,j})$ and $\text{cocharge}_{i,j}(S)=\text{cocharge}_{i,j}(T)$. It remains to consider $\text{cocharge}_{j,k}$. In fact, it could happen that $\text{cocharge}_{j,k}(S)\neq\text{cocharge}_{j,k}(T)$. However, we claim that this is only possible when both integers are at least $M_{j,k}$, so that \begin{equation}\min\{M_{j,k},\text{cocharge}_{j,k}(S)\}=M_{j,k}=\min\{M_{j,k},\text{cocharge}_{j,k}(T)\}.\end{equation} We restrict $S$ and $T$ to entries $x$ with $j\leq x\leq k$ and we consider how the entries move when we rectify these tableaux. By Theorem \ref{thm:jdt}, the rectification does not depend on the order of choices of inside corners, so let us begin by performing jeu de taquin slides into all inside corners that are not on the first row to produce tableaux $S'$ and $T'$. Because $S$ and $T$ differ only in their first row, that is $S_{i',j'}=T_{i',j'}$ for all $i'\geq 2$, we also have $S'_{i',j'}=T'_{i',j'}$ for all $i'\geq 2$. Also note that the $n_j$ $j$'s of $S'$ and $T'$ that are not on the first row must now be on the second row, or in other words $S'_{2,j'}=T'_{2,j'}=j$ for $1\leq j'\leq n_j$. \\

Now repeatedly perform jeu de taquin slides on $S'$ until we obtain a skew tableau $S''$ of shape $\sigma/(n_j)$ for some $\sigma$ and define $T''$ similarly. Let $t_0$ be the number of jeu de taquin slides performed on $S'$ to produce $S''$, so that $t+t_0$ is the number of slides performed on $T'$ to produce $T''$. We have two cases to consider.\\

If $T'_{2,t'}>T'_{1,t'+t+t_0}$ for all $n_j<t'\leq \sigma_2$, then no entry of $T'$ moves down in this step, and because $T'_{2,t'}=S'_{2,t'}$ and $T'_{1,t'+t+t_0}=S'_{1,t'+t_0}$, no entry of $S'$ moves down either. Now because $S''_{2,n_j}=T''_{2,n_j}=j$, when we perform the final $n_j$ jeu de taquin slides, the entries of the first rows of $S''$ and $T''$ do not move, and because $S''_{i',j'}=T''_{i',j'}$ for all $i'\geq 2$, we now have 

\begin{equation}(S\vert_{j,k})_{i',j'}=\begin{cases} j & \text{ if }i'=1\text{ and }j'\leq t,\\ (T\vert_{j,k})_{1,j'-t} & \text{ if }i'=1\text{ and }j'>t, \\ (T\vert_{j,k})_{i',j'} & \text{ otherwise.}\end{cases}\end{equation}

In particular, we have $f(S\vert_{j,k})=f(T\vert_{j,k})+t$ and $\text{cocharge}_{j,k}(S)=\text{cocharge}_{j,k}(T)$. As an illustration, for $j=3$ and $k=5$, these stages in the rectification of the tableaux $S_1$ and $T_1$ are given below.

\begin{align*} S_1&=\tableau{3&4&5&5&5\\2&3&3&4&4&5\\1&1&1&2&3&3&3&3&\textcolor{red}3&\textcolor{red}3&5} \ &T_1&=\tableau{3&4&5&5&5\\2&3&3&4&4&5\\\textcolor{red}1&\textcolor{red}1&1&1&1&2&3&3&3&3&5} \\ \\ S_1'&=\tableau{4&5&5&5\\3&3&3&4&4&5\\&&&&3&3&3&3&\textcolor{red}3&\textcolor{red}3&5} &T_1'&=\tableau{4&5&5&5\\3&3&3&4&4&5\\&&&&&&3&3&3&3&5}\\ S_1''&=\tableau{4&5&5&5\\3&3&3&4&4&5\\&&&3&3&3&3&\textcolor{red}3&\textcolor{red}3&5} &T_1''&=\tableau{4&5&5&5\\3&3&3&4&4&5\\&&&3&3&3&3&5} \\ S_1\vert_{3,5}&=\tableau{5&5\\4&4&4&5&5\\3&3&3&3&3&3&3&\textcolor{red}3&\textcolor{red}3&5} &T_1\vert_{3,5}&=\tableau{5&5\\4&4&4&5&5\\3&3&3&3&3&3&3&5} \end{align*}

Now suppose that $T'_{2,t'}\leq T'_{1,t'+t+t_0}$ for some $t'$ with $n_j<t'\leq\sigma_2$. Then when we perform jeu de taquin slides to produce $T''$, an entry $x>j$ of the second row must move down on some slide. On subsequent slides, because the second row of $T'$ was weakly increasing, entries of the second row will continue to move down, and we will have $j<T''_{2,t'}\leq T''_{1,t'+1}$ for some $t'>n_j$. Meanwhile, we must have $j<S''_{2,t'}\leq S''_{2,t'+t}$ for some $t'>n_j$ because if we delete $t$ $j$'s on the first row of $S''$ and then perform $t$ jeu de taquin slides, we would obtain $T''$ so some entry of the second row must move down.\\

When we perform the final $n_j$ jeu de taquin slides, the entries of the first rows of $S''$ and $T''$ do not move, the $n_j$ $j$'s in the second rows of $S''$ and $T''$ each move one cell down, and the remaining entries in the second rows move at most $n_j$ cells to the left. Therefore, we now have $f(T\vert_{j,k})\leq n_j$ and $f(S\vert_{j,k})\leq n_j+t$. By Proposition \ref{prop:caterpillar}, Part 3 and Part 4, we have $n_j\leq (M-1)+\alpha_2+\cdots+\alpha_{j-1}\leq \alpha_j-M_{j,k}$, and therefore
$\text{cocharge}_{j,k}(S),\text{cocharge}_{j,k}(T)\geq M_{j,k}$.\\

As an illustration, for $j=3$ and $k=5$, these stages in the rectification of $S_2$ and $T_2$ are given below. When we rectify $T_2'$, a $5$ moves down from the second row so the second rows of $S_2''$ and $T_2''$ will be different. However, when we rectify $S_2''$ and $T_2''$, these $5$'s move at most three cells to the left, so $f(S_2\vert_{3,5})\leq 5$ and $f(T_2\vert_{3,5})\leq 3$, which means that $\text{cocharge}_{3,5}(S),\text{cocharge}_{3,5}(T)\geq 4$. Informally, the only way that $\text{cocharge}_{j,k}(S)$ could differ from $\text{cocharge}_{j,k}(T)$ is if a cell in the second row of $T$ moves down prematurely, but if this happens, then $f(T\vert_{j,k})$ will be small enough to make $\text{cocharge}_{j,k}(T)\geq M_{j,k}$.

\begin{align*} S_2&=\tableau{3&5&5\\2&3&3&4&4&5&5&5\\1&1&1&2&3&3&3&3&\textcolor{red}3&\textcolor{red}3&5} \ &T_2&=\tableau{3&5&5\\2&3&3&4&4&5&5&5\\\textcolor{red}1&\textcolor{red}1&1&1&1&2&3&3&3&3&5} \\ \\ S_2'&=\tableau{5&5\\3&3&3&4&4&5&5&5\\&&&&3&3&3&3&3&3&5} &T_2'&=\tableau{5&5\\3&3&3&4&4&5&5&5\\&&&&&&3&3&3&3&5}\\  S_2''&=\tableau{5&5\\3&3&3&4&4&5&5&\textcolor{red}5\\&&&3&3&3&3&3&3&\textcolor{red}5} &T_2''&=\tableau{5&5\\3&3&3&4&4&5&\textcolor{red}5\\&&&3&3&3&3&\textcolor{red}5&5} \\ S_2\vert_{3,5}&=\tableau{5&5\\4&4&5&5&\textcolor{red}5\\3&3&3&3&3&3&3&3&3&\textcolor{red}5} &T_2\vert_{3,5}&=\tableau{5&5\\4&4&5&\textcolor{red}5\\3&3&3&3&3&3&3&\textcolor{red}5&5} \end{align*}

In summary, the map $\varphi$ is a bijection and it satisfies $\text{cocharge}_{\Pi''}(S)=\text{cocharge}_{\Pi'}(\varphi(S))$. Also note that by definition we have
\begin{equation}\text{cocharge}_\Pi(T)=\begin{cases}\text{cocharge}_{\Pi'}(T)&\text{ if }\text{cocharge}_{1,j}(T)\leq M-1,\\\text{cocharge}_{\Pi'}(T)+1&\text{ if }\text{cocharge}_{1,j}(T)\geq M.\end{cases}\end{equation}
Therefore, using our induction hypothesis, we have \begin{align}G_{\bm\lambda}(\bm x;q)&=q G_{\bm\lambda'}(\bm x;q)+(1-q) G_{\bm\lambda''}(\bm x;q)\\\nonumber&=q\sum_{T\in\text{SSYT}(\alpha)}q^{\text{cocharge}_{\Pi'}(T)}s_{\text{shape}(T)}+(1-q)\sum_{S\in\text{SSYT}(\beta)}q^{\text{cocharge}_{\Pi''}(S)}s_{\text{shape}(S)}\\\nonumber&=q\sum_{\substack{T\in\text{SSYT}(\alpha)\\\text{cocharge}_{1,j}(T)\geq M}}q^{\text{cocharge}_{\Pi'}(T)}s_{\text{shape}(T)}\\\nonumber&+q\sum_{\substack{T\in\text{SSYT}(\alpha)\\\text{cocharge}_{1,j}(T)\leq M-1}}q^{\text{cocharge}_{\Pi'}(T)}s_{\text{shape}(T)}+(1-q)\sum_{S\in\text{SSYT}(\beta)}q^{\text{cocharge}_{\Pi''}(S)}s_{\text{shape}(S)}\\\nonumber&=q\sum_{\substack{T\in\text{SSYT}(\alpha)\\\text{cocharge}_{1,j}(T)\geq M}}q^{\text{cocharge}_{\Pi'}(T)+1}s_{\text{shape}(T)}+\sum_{\substack{T\in\text{SSYT}(\alpha)\\\text{cocharge}_{1,j}(T)\leq M-1}}q^{\text{cocharge}_{\Pi'}(T)}s_{\text{shape}(T)}\\\nonumber&=\sum_{T\in\text{SSYT}(\alpha)}q^{\text{cocharge}_\Pi(T)}s_{\text{shape}(T)}.
\end{align}This completes the proof.\end{proof}

Theorem \ref{thm:caterpillar} expresses the LLT polynomial $G_{\bm\lambda}(\bm x;q)$ in terms of the weighted graph $\Pi(\bm\lambda)$ but not in terms of $\bm\lambda$ itself. In other words, if $\Pi(\bm\lambda)$ and $\Pi(\bm\mu)$ are equal and triangle-free, then $G_{\bm\lambda}(\bm x;q)=G_{\bm\mu}(\bm x;q)$. We conjecture that the formula \eqref{eq:caterpillar} does not depend on the labelling of the vertices of $\Pi(\bm\lambda)$, provided that whenever $i<j<k$ and $v_i$ is adjacent to $v_k$, then $M_{j,k}=|v_j|$. By Proposition \ref{prop:caterpillar}, Part 5, this is equivalent to the following statement.

\begin{conjecture} Let $\bm\lambda$ and $\bm\mu$ be horizontal-strips. If the weighted graphs $\Pi(\bm\lambda)$ and $\Pi(\bm\mu)$ are isomorphic and triangle-free, then the LLT polynomials $G_{\bm\lambda}(\bm x;q)$ and $G_{\bm\mu}(\bm x;q)$ are equal. \end{conjecture}

We further conjecture that in general a horizontal-strip LLT polynomial $G_{\bm\lambda}(\bm x;q)$ is determined by its unlabelled weighted graph.

\begin{conjecture} \label{con:dependsongraph} Let $\bm\lambda$ and $\bm\mu$ be horizontal-strips. If the weighted graphs $\Pi(\bm\lambda)$ and $\Pi(\bm\mu)$ are isomorphic, then the LLT polynomials $G_{\bm\lambda}(\bm x;q)$ and $G_{\bm\mu}(\bm x;q)$ are equal. \end{conjecture}


\begin{thebibliography}{10}

\bibitem{lltunicellepos}
{\sc P.~Alexandersson}, {\em{LLT polynomials, elementary symmetric functions and melting lollipops}}. J. Algebraic Combin. (2020).

\bibitem{lltcombe}
{\sc P.~Alexandersson and R.~Sulzgruber}, {\em{A combinatorial expansion of vertical-strip LLT polynomials in the basis of elementary symmetric functions}}. arXiv:2004.09198 (2020).

\bibitem{lltmpath}
{\sc P.~Alexandersson and J.~Uhlin}, {\em{Cyclic sieving, skew Macdonald polynomials and Schur positivity}}. arXiv:1908.00083 (2019).

\bibitem{shuffle}
{\sc E.~Carlsson and A.~Mellit}, {\em{A proof of the shuffle conjecture}}. J. Amer. Math. Soc. 31 661--697 (2018).

\bibitem{lltepos}
{\sc M.~D'Adderio}, {\em{e-Positivity of vertical strip LLT polynomials}}. arXiv:1906.02633 (2019).

\bibitem{youngtab}
{\sc W.~Fulton}, {\em {Young Tableaux: With Applications to Representation Theory and Geometry}}. Cambridge University Press (1997).

\bibitem{lltpos}
{\sc I.~Grojnowski and M.~Haiman}, {\em{Affine Hecke algebras and positivity of LLT and Macdonald polynomials}}. (2007).

\bibitem{combmacdonald}
{\sc J.~Haglund, M.~Haiman, and N.~Loehr}, {\em{A combinatorial formula for Macdonald polynomials}}. J. Amer. Math. Soc. 18 (2005) 735--761.

\bibitem{combdiag}
{\sc J.~Haglund, M.~Haiman, N.~Loehr, J.~Remmel, and A.~Ulyanov}, {\em{A combinatorial formula for the character of the diagonal coinvariants}}. Duke Math. J. 126 195--232 (2005).

\bibitem{lltunicellschur}
{\sc J.~Huh, S.~Nam, and M.~Yoo}, {\em{Melting lollipop chromatic quasisymmetric functions and Schur expansion of unicellular LLT polynomials}}. arXiv:1812.03445 (2018).

\bibitem{lltoriginal}
{\sc A.~Lascoux, B.~Leclerc, and J-Y.~Thibon}, {\em{Ribbon tableaux, Hall--Littlewood functions, quantum affine algebras, and unipotent varieties}}. J. Math. Phys. 38 (1997), no. 2, 1041--1068.

\bibitem{plactic}
{\sc A.~Lascoux, B.~Leclerc, J-Y.~Thibon}, {\em{The plactic monoid}}. Encyclopedia of Mathematics and its Applications (2002) vol. 90, 164--196.

\bibitem{charge}
{\sc A.~Lascoux and M-P.~Sch\"{u}tzenberger}, {\em{Sur une conjecture de H. O. Foulkes}}. C. R. Acad. Sci. 286 no. 7, (1978).

\bibitem{lrkl}
{\sc B.~Leclerc and J-Y.~Thibon}, {\em{Littlewood--Richardson coefficients and Kazhdan--Lusztig polynomials}}. Adv. Stud. Pure Math. 28 155--220 (2000).

\bibitem{lltkschurband}
{\sc C.~Miller}, {\em{On the $k$-Schur Positivity of $k$-Bandwidth LLT Polynomials}}. (2019).

\bibitem{enum2}
{\sc R.~Stanley}, {\em {Enumerative Combinatorics. Vol. 2}}. Cambridge University Press (1999).
\end{thebibliography}
\end{document}